\documentclass[10pt]{article}
\usepackage{graphicx} 
\usepackage{amsmath, amsthm, amssymb}
 \usepackage{fullpage}
\usepackage{dsfont}
\usepackage{natbib}
\usepackage{color}
\usepackage{appendix}

\theoremstyle{thmstyleone}%
\newtheorem{theorem}{Theorem}
\theoremstyle{thmstyletwo}%
\newtheorem{example}{Example}%
\newtheorem{remark}{Remark}%
\theoremstyle{thmstylethree}%
\newtheorem{definition}{Definition}
\newtheorem{lemma}{Lemma} 
\usepackage{dsfont}
\newtheorem{corollary}{Corollary}

\DeclareMathOperator*{\argmax}{arg\,max}

\newcommand{\bU}{\boldsymbol{V}}

\newcommand{\bX}{\boldsymbol{X}}

\newcommand{\bG}{\boldsymbol{G}}

\newcommand{\bw}{\boldsymbol{w}}
\newcommand{\bx}{\boldsymbol{x}}

\newcommand{\ba}{\boldsymbol{a}}

\newcommand{\bzero}{\boldsymbol{0}}

\newcommand{\bGamma}{\boldsymbol{\Gamma}}
\newcommand{\bSigma}{\boldsymbol{\Sigma}}
\newcommand{\bbeta}{\boldsymbol{\beta}}
\newcommand{\btheta}{\boldsymbol{\theta}}
\newcommand{\stheta}{\btheta^\star_{\model}}
\newcommand{\wtheta}{\widehat{\btheta}_{\model,n}}

\newcommand{\blambda}{\boldsymbol{\lambda}}
 
 \newcommand{\bvartheta}{\boldsymbol{\vartheta}}

\newcommand{\bnu}{\boldsymbol{\nu}}
\newcommand{\bzeta}{\boldsymbol{\zeta}}
\newcommand{\sample}{\overline{\bX}}
\DeclareMathOperator*{\argmin}{arg\,min}

\newcommand{\model}{\textbf{m}}

\newcommand{\trueb}{\model,n,p,\stheta}
\newcommand{\true}{\model,n,p,\stheta,b}
\newcommand{\est}{\model,n,p,\wtheta,b}
\newcommand{\estall}{\model,n,p,\wtheta}
\newcommand{\trueall}{\model,n,p,\stheta}
\newcommand{\any}{\model,n,p,\btheta,b}
\newcommand{\anyb}{\model,n,p,\btheta}

\newcommand{\alt}{(\Upsilon n_b^{-1/2})}
\newcommand{\altnop}{\Upsilon n_b^{-1/2}}

\newcommand{\blamany}{\boldsymbol{\lambda}_{\any}}
\newcommand{\bZ}{\boldsymbol{Z}}

\newcommand{\bZany}{\boldsymbol{Z}_{\any}}
\newcommand{\bS}{\boldsymbol{S}}
\newtheorem{assumptions}{Assumptions}
 \usepackage{authblk} 
 \title{Goodness-of-fit testing of the distribution of posterior classification probabilities for validating model-based clustering}

\author[1]{Salima El Kolei}
\author[2]{Matthieu Marbac}

\affil[1]{Univ. Rennes, Ensai, CNRS, CREST-UMR 9194, 35000 Rennes, France}
\affil[2]{Université Bretagne Sud, UMR CNRS 6205, LMBA, F-56000 Vannes, France.}

\begin{document}
 
\maketitle

\abstract{We present the first method for assessing the relevance of a model-based clustering result in a general framework.  Standard validation criteria, like the adjusted Rand index, rely on external labels to assess partition accuracy; consequently, they are inapplicable to real-world clustering problems where labels are missing. In contrast, our method offers an internal goodness-of-fit diagnostic, since it evaluates the validity of the clustering mechanism by testing the specification of the posterior probabilities of classification defined on the unit simplex. Because this simplex dimension is fixed by the number of clusters, the procedure naturally circumvents the curse of dimensionality, making it applicable to high-dimensional data where traditional density-based tests fail.  The testing procedure requires only a consistent estimator of the parameters and the associated posterior classification probabilities for each observation, and its implementation is straightforward, as no additional model fitting is needed. Under the null hypothesis, the method exploits the fact that any functional transformation of the posterior probabilities has the same expectation under both the model being tested and the true data-generating process. The resulting goodness-of-fit test is constructed via an empirical likelihood approach with a growing number of moment conditions, allowing asymptotic detection of any alternative.  A block-splitting strategy, employed to account for parameter estimation, provides a vector of test statistics that behave like a vector of independent chi-square random variables. Therefore, the goodness-of-fit of the posterior classification probabilities is assessed via the goodness-of-fit of the vector of empirical likelihood ratio test statistics. Hence, based on the distribution of this vector of statistics, different goodness-of-fit tests (e.g., Kolmogorov-Smirnov) can be used to investigate the distribution of the vector of test statistics with an exact asymptotic significance level\\
\textbf{keywords:} Clustering; Empirical likelihood;  Estimating equations;  Growing number of equations; Goodness-of-fit; Mixture models;
}

\section{Introduction}
Model-based clustering enables clustering by estimating the distribution of the observed variables using a finite mixture model \citep{McL00,compiani2016using,Fruhwirth2019handbook,chen2023statistical}. In this approach, subjects generated from the same mixture component are considered to belong to the same cluster. Unlike non-model-based clustering methods, which primarily aim to estimate a partition, model-based clustering allows for the estimation of posterior classification probabilities, thereby capturing the uncertainty in cluster assignments. As a result, with model-based clustering, a partition can be estimated, and the risk of misclassification can be computed for each observation. This framework assumes the existence of a $d$-dimensional random variable $\bX\in\mathcal{X}$ and a latent variable $\bU$ defined on ${1,\ldots,K}$, where $K$ is the number of clusters. The variables $\bU$ and $\bX$ are assumed to be dependent, and since $\bU$ is not observed, the marginal distribution of the observed data $\bX$ follows a mixture model with $K$ components. Consequently, model-based clustering estimates the distribution of $\bX$ and thus achieves clustering by estimating the posterior classification probabilities (\emph{i.e.,} the conditional distribution of $\bU$ given $\bX$).

It is crucial to distinguish the validation of the estimated posterior probabilities from the assessment of clustering performance. In the clustering literature, indices such as the Adjusted Rand Index (ARI) or Normalized Mutual Information (NMI) are standard for evaluating the quality of a partition. However, these are \emph{external validation} measures that require knowledge of the ground-truth labels. While these metrics are commonly used in simulation studies to compare the performance of different clustering algorithms, they are, by definition, unavailable in real-world unsupervised applications. Consequently, they cannot be used as a diagnostic tool to assess whether a chosen model is appropriate for a specific dataset. Our work addresses this challenge by focusing on \emph{internal validation}, specifically the goodness-of-fit of the latent structure itself. Unlike global density-based tests that assess the distribution of $\bX$, we propose to test whether the classification mechanism represented by the posterior probabilities is well-specified. This provides a formal diagnostic framework that remains operational when no ground truth is available and targets the clustering objective rather than the overall modeling fit.

The distribution of $\bX$ is specified by a model $\model=\{K,\mathcal{F}\}$, which defines the number of components $K$ and a family of component distributions $\mathcal{F}$. Hence, for a given model $\model$, the set of densities is
\begin{equation}\label{eq:mixturemodel}
\mathcal{G}_\model = \left\{ g_{\model,\btheta}(\cdot) = \sum_{k=1}^{K} \pi_k f_k(\cdot;\bvartheta_k)
,\, (f_1,\ldots,f_{K})\in\mathcal{F}  \text{ and } \btheta=(\boldsymbol{\pi}^\top,\bvartheta_1^\top,\ldots,\bvartheta_{K}^\top)^\top \in \Theta_{\model} \right\},
\end{equation}
where $f_k$ denotes the density of component $k$, defined such that the set of $K$ component densities belongs to the space $\mathcal{F}$. The parameter vector $\btheta=(\boldsymbol{\pi}^\top,\bvartheta_1^\top,\ldots,\bvartheta_{K}^\top)^\top$ groups all model parameters and belongs to the space $\Theta_{\model}$, which depends on $\model$. The component proportions satisfy $0<\pi_k<1$ and $\sum_{k=1}^{K}\pi_k=1$, meaning that the proportion vector $\boldsymbol{\pi}=(\pi_1,\ldots,\pi_{K})^\top$ is defined on a simplex of size $K$. This definition allows $\Theta_{\model}$ to be either finite- or infinite-dimensional, thus accommodating both parametric and nonparametric approaches. In a parametric framework, $\mathcal{F}$ can be the space defined as the product of $K$ subspaces, each composed of all $d$-variate Gaussian densities. In this case, \eqref{eq:mixturemodel} defines the set of all Gaussian mixture models \citep{banfield1993model}, with $\bvartheta_k$ specifying the mean and covariance matrix of component $k$. If parsimonious Gaussian mixtures are considered, such as isotropic Gaussian mixture models or mixtures of factor analyzers \citep{mcnicholas2008parsimonious}, the model $\model$ imposes constraints on $\Theta_{\model}$. Other standard parametric mixture models covered by \eqref{eq:mixturemodel} include skewed mixture models \citep{wallace2018variable}, beta mixture models \citep{ji2005applications}, and gamma mixture models \citep{mayrose2005gamma}. Nonparametric approaches are also encompassed by \eqref{eq:mixturemodel}. For instance, $\mathcal{F}$ can be defined as a product space of $K$ subspaces $\mathcal{F}_k$, where $f_k \in \mathcal{F}_k$ and $\mathcal{F}_k$ is the set of $d$-variate densities defined as products of univariate densities \citep{HettmanspergerJRSSB2000,HallAOS2003}. In this case, $\bvartheta_k$ is an infinite-dimensional parameter specifying all univariate densities for component $k$. Finally, \eqref{eq:mixturemodel} also covers semiparametric approaches. For example, location mixtures \citep{hunter2007inference} fall within this framework, as $\mathcal{F}$ allows all components to share the same symmetric density function up to a translation. In this case, each $\bvartheta_k$ specifies both the translation parameter and the common symmetric density. Thus, $\bvartheta_k$ contains an infinite-dimensional component (the symmetric density) shared across all components, as well as a finite-dimensional component (the scalar defining the translation) without constraints across components. The distributions defined by \eqref{eq:mixturemodel} can also accommodate complex spaces $\mathcal{X}$, including functional data \citep{Bouveyron}, partially observed data \citep{miaoJASA2016}, mixed-type data \citep{marbac2017model}, tensor data \citep{mai2022doubly}, and extreme data \citep{tendijck2023modeling}.

For most applications, the true model $\model^\star$ is unknown. Therefore, the standard procedure seeks to identify the best model $\widehat\model$ from the data, chosen among a finite collection of models $\mathcal{M}=\{\model_1,\model_2,\ldots\}$. This problem is inherently complex due to the nature of clustering itself, which does not allow for the selection of $\widehat\model$ based on the accuracy of posterior probability estimates. Unlike in supervised or semi-supervised classification, prediction error rates cannot be directly assessed since the realizations of the latent cluster membership variable $\bU$ are unavailable. As a result, model selection is often guided by the model’s ability to capture the distribution of the observed variables, even though the fundamental objective remains the estimation of posterior probabilities of classification. In a parametric framework, model selection can be achieved using likelihood ratio tests or information criteria. For instance, \citet{Chen} proposes a homogeneity test to determine whether $K=1$ in a Gaussian mixture model using a likelihood ratio test. Alternatively, leveraging the control of the log-likelihood ratio obtained by \citet{DacunhaAOAS1999} through locally conic parameterization, \citet{KeribinSan2000} demonstrates the consistency of likelihood-based penalization criteria, including the Bayesian Information Criterion (BIC) \citep{Schwarz:78}. Other approaches within the parametric framework include those proposed by \citet{james2001consistent} and \citet{woo2006robust}. In a nonparametric framework, model selection efforts have primarily focused on determining the number of components in a mixture model where each component is defined as a product of univariate densities. These approaches often rely on estimating the rank of a specific matrix \citep{KasaharaJRSSB2014,BonhommeJRSSB2016} or a specific operator \citep{Kwon2019estimation}. Recently, \citet{Chaumaray} extended model selection for this class of mixture models by addressing the challenge of feature selection, which imposes constraints on $\mathcal{F}$. In all the methods mentioned above, authors focus on the consistency of $\widehat\model$ under the assumption that the true model belongs to the set of candidate models (\emph{i.e.,} $\model^\star\in\mathcal{M}$). Therefore, these methods allow for the selection of the \emph{best} model among the competing ones but do not assess whether this model is equal (or at least close) to the true model. Consequently, they provide no information on the relevance of the selected model. In particular, they cannot guarantee that the estimated uncertainty in cluster assignments quantified by the posterior probabilities of classification faithfully represents the underlying data structure.  For instance, in the parametric case, they do not allow for evaluating the validity of the parametric assumptions underlying $\widehat\model$, whereas in the nonparametric case, they do not assess the assumption of independence between variables within components or the assumption of symmetry within components.

To investigate the relevance of $\widehat\model$, goodness-of-fit methods such as the Kolmogorov–Smirnov test \citep{massey1951kolmogorov} or the Cramer-von Mises test \citep{darling1957kolmogorov} could be considered. However, since the goal is to test only the relevance of $\widehat\model$, the null hypothesis is composite. To implement these tests, the parameters must be specified. After the unknown parameters are estimated from the entire data set, \citet{braun1980simple} proposes a procedure where the transformed sample is randomly partitioned into a large number of groups, and a goodness-of-fit statistic is calculated for each group. Note that the number of groups is determined according to the rate of convergence of the estimator of the parameter computed on the entire data set. These statistics are then used to construct a test which, asymptotically, can attain any desired level, and which requires only standard tables of critical values for its implementation. Alternatively, goodness-of-fit testing can be achieved using empirical likelihood \citep{baggerly}. Among this family of methods, one can highlight the contribution of \citet{Peng}, who generalize the empirical likelihood approach \citep{owen2001empirical} to allow the number of constraints to grow with the sample size and for the constraints to use estimated criterion functions. Allowing the number of constraints to grow with the sample size enables the detection of any alternative asymptotically. Again, since the null hypothesis is composite, a value for the parameters must be considered. If the asymptotic distribution of the maximum of the empirical likelihood ratio on $\btheta$ over $\Theta_{\model}$ is established by \citet{QinAOAS1994}, the procedure becomes complex for mixture models since finding the parameters that maximize the empirical likelihood is challenging. Moreover, estimating the parameters of a mixture model is often done via iterative algorithms such as the EM algorithm \citep{mclachlan2008algorithm} or the MM algorithm \citep{LevineBiometrika2011}, which can be computationally intensive. Therefore, it would be desirable for the testing procedure not to require any additional parameter estimation. Alternatively, \citet{BAGKAVOS2023107815} propose a goodness-of-fit procedure that uses the maximum likelihood estimates of normal mixture densities with a known number of components. All of the goodness-of-fit procedures mentioned above suffer from two main drawbacks: their power decreases drastically as the dimension of $\bX$ increases, and they do not directly address the clustering goal (\emph{i.e.,} modeling the distribution of the conditional probabilities of classification).

In this paper, we present the first method that permits validating a model-based clustering procedure by directly focusing on the adjustment of the conditional distribution of the latent variable $\bU$ given the observed variable $\bX$ (\emph{i.e.,} the posterior probabilities of classification). Note that, whatever the nature of the observed variable $\bX$, the conditional probabilities of classification are always defined on a simplex of size $K$. Therefore, the proposed method can be used for any type of data $\bX$. To be applied, the testing procedure only requires  a consistent estimator of the model parameters as well as its associated conditional probabilities of classification for each observation. Thus, the implementation of the testing procedure is straightforward since it does not require any additional estimator. Indeed, it only considers the parameter estimators preliminarily assessed during the model-based clustering step. Hence, the usual estimation algorithms can be used for clustering. However, the procedure requires knowledge of the convergence rate of such parameter estimators. Under the null hypothesis, the method relies on the fact that any functional transformation of the posterior probabilities of classification has the same expectation with respect to both the model being tested and the true model. The use of functional moments permits circumventing the fact that nothing is known about the true distribution of $(\bU^\top,\bX^\top)^\top$ except that the observed data are generated from it. Hence, the empirical mean of any functional transformation of the posterior probabilities of classification can consistently estimate the expectation under the true model, while the expectation under the model being tested can be easily assessed via Monte Carlo methods. The goodness-of-fit testing is achieved with an empirical likelihood method that considers a number of moment conditions increasing with the sample size in order to asymptotically detect any alternative. Since the procedure uses an estimator of the parameters $\widehat{\btheta}_{\model,n}$ previously estimated, the empirical log-likelihood ratio does not converge to a chi-square distribution. To circumvent this issue, the method proposes performing data splitting into blocks, and the empirical log-likelihood ratio is computed for each block of data.  We show that the vector composed of the empirical likelihood ratio test statistics obtained on each block of data behaves like a vector of independent chi-square random variables. Hence, different goodness-of-fit methods (\emph{e.g.,} Kolmogorov-Smirnov, maximum statistics, etc.) can be used to evaluate its distribution and, consequently, to test the relevance of the posterior classification probabilities.

The methodological contribution of this paper is to propose a validation method for model-based clustering, directly focusing on the posterior probabilities of classification. Developing such a method entails new theoretical advancements. Indeed, hypothesis testing procedures relying on empirical likelihood with a growing number of moment conditions generally consider the true parameters \citep{Hjort2009AOS,Peng}. Here, we extend this family of testing procedures by considering a consistent estimator of the model parameter, including the case of infinite-dimensional parameters. The extension of the empirical likelihood method we propose uses some elements of the goodness-of-fit procedure performed with parameter estimation as proposed by \citet{braun1980simple}. Furthermore, an extension of this procedure is proposed in this paper, as we do not restrict the situation to parametric distributions, thereby extending this family of approaches to infinite-dimensional parameters.  
Note that if a method allows for infinite-dimensional parameters, it requires a specific rate of convergence for its estimator. To the best of our knowledge, establishing the convergence rate in supremum norm for nonparametric mixture model estimators remains an open problem. While these rates are necessary to satisfy the method's theoretical requirements, their derivation lies beyond the scope of this paper. 
 
To paper is organized as follows. Section~\ref{sec:proc} describes the {goodness-of-fit procedure that investigates the relevance of model-based clustering results. Section~\ref{sec:theo} presents the theoretical guarantees of the procedure including the control of the level of the procedure. Section~\ref{sec:num} illustrates the relevance of the approach on numerical experiments. Section~\ref{sec:concl} gives a conclusion. The proofs are given in the Supplementary Material. 
 The proposed method is implemented in the R package GOFclustering available on CRAN.

\paragraph{Notations.}

Throughout the paper, we adopt the following notational conventions. The symbol $^{\star}$ always denotes the true value of a parameter under consideration: for instance, $\theta^{\star}_{\model}$ represents the true parameter under the model $\model$. The hat symbol indicates estimators computed from the observed sample: for example, $\widehat{\theta}_{\model,n}$ denotes an estimator of the true parameter $\theta^{\star}_{\model}$ computed from $n$ observations. 
For a matrix $A \in \mathcal{M}_p(\mathbb{R})$, we denote by $\|A\|_{sp}$ its spectral norm, by $\sigma_1(A)$ its smallest singular value, and by $\|A\|_{\max}$ the maximum absolute value of its entries.
Finally, we write $T^{\dagger}$ for the maximum over $b=1,\dots,B$ of the scalar magnitude of $T_b$, that is absolute value for real quantities, Euclidean norm for vectors, spectral norm for matrices.
For clarity, we also recall the stochastic order notation. For a sequence of random variables $X_n$ and a sequence of positive constants $a_n$, we write $X_n = O_p(a_n)$ if $X_n / a_n$ is bounded in probability, and $X_n = o_p(a_n)$ if $X_n / a_n \xrightarrow{\mathbb{P}} 0$ as $n \to \infty$. These conventions are used consistently throughout the paper.

\section{Goodness-of-fit testing of the conditional probabilities of classification}  \label{sec:proc}
\subsection{Conditional probabilities of classification}
The true distribution of the observed variables $\bX$ is specified by a particular model $\model^\star$ and a particular parameter $\btheta^\star=(\boldsymbol{\pi}^{\star},\bvartheta^{\star}_1,\ldots,\bvartheta^{\star}_{K})\in\Theta_{\model^\star}$ that defines the $K$-component mixture model with the probability density or mass function
$$g_{\model^\star,\btheta^\star}(\bx)=\sum_{k=1}^{K} \pi_k^\star f_{k}^\star(\bx;\bvartheta^{\star}_k),$$
where $\boldsymbol{\pi}^{\star}=(\pi_1^\star,\ldots,\pi_K^\star)$, $\pi_k^\star$ denotes the proportions of component $k$ with $0<\pi_k^\star\leq 1$ and $f_{k}^\star(\bx;\bvartheta^{\star}_k)$ denotes the  probability density or mass function  of component $k$ evaluated at $\bx$ with parameter $\bvartheta^{\star}_k$. 

  In model-based clustering,    the aim is to fit $c_{\model,\btheta,k}(\bx)$ that corresponds to the conditional probability that $\bU=k$ given $\bX=\bx$ obtained under the model defined by $\{\model,\btheta\}$, where $\bU\in\{1,\ldots,K\}$ indicates the cluster membership.  Note that it is standard to define a cluster as the subset of observations arising from the same component of the mixture model; in such a case, $\bU$ directly defines the cluster membership.   For any model $\model$ and parameter $\btheta$, we define $c_{\model,\btheta}(\bX)=(c_{\model,\btheta,1}(\bX),\ldots,c_{\model,\btheta,K}(\bX))^\top$ as the vector composed of the conditional probabilities of component memberships given $\bX$, specified by model $\model$ with parameter $\btheta$, leading that
\begin{equation}\label{eq:odds}
\forall \bx\in\mathcal{X},\; \forall k \in\{1,\ldots,K\},\; c_{\model,\btheta,k}(\bx) = \frac{\pi_k f_k(\bx;\bvartheta_k)}{\sum_{\ell=1}^K \pi_{\ell} f_{\ell}(\bx;\bvartheta_{\ell})}.
\end{equation}

For clarity, $c_{\model,\btheta}(\bx)$ denotes the mapping defined in \eqref{eq:odds}, while $c_{\model,\btheta}(\bX)$ denotes the corresponding random vector taking values in the simplex $S^K$. 

The aim of model-based clustering is to estimate the function giving the posterior probabilities of classifications $c_{\model,\btheta}$ from an observed sample composed of independent realizations of $\bX$. This estimation is achieved by selecting the best model and estimating its parameters from the observed sample. From the vector $c_{\model,\btheta}(\bx)$, a hard clustering can be achieved by assigning any $\bx$ to its most likely cluster (\emph{i.e.,} the component of $c_{\model,\btheta}(\bx)$ with the largest value). In addition, the uncertainty associated with this classification rule can be obtained by considering the probability masses of the components of $c_{\model,\btheta}(\bx)$ that differ from the cluster assignment.

\subsection{On the notion of well-specification of a distribution for clustering}

The distribution defined by the model $\model$ and the parameters $\btheta\in\Theta_{\model}$ is said to be \emph{well-specified for the data distribution} \citep{white1982maximum} if
\begin{equation}\label{eq:welldist}
\forall \bx\in\widetilde{\mathcal{X}},\;  g_{\model,\btheta}(\bx)=g_{\model^\star,\btheta^\star}(\bx),
\end{equation}
where $\widetilde{\mathcal{X}}\subseteq \mathcal{X}$ is equal to $\mathcal{X}$ up to a subspace of null measure. 
The notion of well-specification for clustering is often conflated with the well-specification of the data density. It is important to note that if a model is well-specified for the data distribution, then the probabilities of classification are well specified (\emph{i.e.,} $g_{\model,\btheta} = g_{\model^\star,\btheta^\star}$ implies $c_{\model,\btheta} \equiv c_{\model^\star,\btheta^\star}$ where $\equiv$ indicates that the equality holds up to a component labelling meaning that there exists a permutation of the coordinates of the vector such that the equalities holds for any $\bx$). However, the converse is not true as illustrated by the following example.
\begin{example}
\label{rem:uninformative}
 Suppose $\bX=(X_1,X_2)^\top$ has a density $g_{\model^\star,\btheta^\star}(\bx)=\sum_{k=1}^2 \frac{1}{2}\phi(x_1; (-1)^k,1) \mathds{1}_{\{-1/2\leq x_2\leq 1/2\}}$ where $\phi(y;m,v)$ denotes the normal density with mean $m$ and variance $v$ evaluated at $y$. And, consider the mixture model defined by $g_{\model,\btheta}(\bx)=\sum_{k=1}^2 \frac{1}{2}\phi(x_1; (-1)^k,1) \phi(x_2;0,1)$. Then, even that $g_{\model,\btheta}$ is not well-specified, we have $c_{\model,\btheta} = c_{\model^\star,\btheta^\star}$. 
\end{example}

This lack of equivalence highlights the primary advantage of our approach: drawing from the literature on probability calibration \citep{dawid1982well}, we focus on the distribution of $c_{\model,\btheta}(\bX)$ rather than on $\bX$ itself to directly evaluate the model's relevance for its intended clustering purpose. Furthermore, while the data $\bX$ may lie in a very high-dimensional or complex space, the posterior probabilities $c_{\model,\btheta}(\bX)$ always take values in the  simplex $S^K$. This dimensionality reduction ensures that our test remains focused on the classification structure, regardless of the nature or complexity of the input features.  
Since this paper aims at testing the relevance of the posterior probabilities of classification, we introduce the notion of weakly well-specification for soft-clustering (\emph{i.e.,} for the classification probabilities). Note that the definition is weaker than a strict equality $c_{\model^\star,\btheta^\star} =c_{\model,\btheta}$, justifying the term of weak well-specification,  while the term soft clustering aims to indicate that the target is the posterior probabilities of classification and not only the hard assignement.


\begin{definition}[Weakly well-specified for soft clustering]
A mixture model defined by $\{\model,\btheta\}$ is weakly well-specified for soft clustering if its posterior probabilities of component memberships $c_{\model,\btheta}(\bX)$ have the same distribution under  $\{\model,\btheta\}$ and under the true model  $\{\model^\star,\btheta^\star\}$, leading that 
\begin{equation}\label{eq:weakcl}
\forall \boldsymbol{c} \in S^K,\; \mathbb{P}_{\model,\btheta}(c_{\model,\btheta}(\bX)\leq \boldsymbol{c})=\mathbb{P}_{\model^\star,\btheta^\star}(c_{\model,\btheta}(\bX)\leq \boldsymbol{c}).
\end{equation}
\end{definition}

 The following example shows that the classification boundary can be well-specified even if the classification probabilities are not weakly well-specified.
 
\begin{example}
Suppose $\bX\in\mathbb{R}$ that follows a bi-component mixture of Student distributions $g_{\model^\star,\btheta^\star}(\bx) = \frac{1}{2}t_3(\bx;-1)+\frac{1}{2}t_3(\bx;1)$ and consider the Gaussian mixture $g_{\model,\btheta}(\bx) = \frac{1}{2}\phi(\bx;-1,1)+\frac{1}{2}\phi(\bx;1,1)$. Note that both models share the same decision boundary at $\bx=0$ due to symmetry, meaning that their MAP  (Maximum A Posteriori) assignments coincide (\emph{i.e.,} $\forall \bx \in\mathcal{X}, \; \argmax_k c_{\model,\btheta,k}(\bx) =\argmax_k c_{\model^\star,\btheta^\star,k}(\bx) $). However,  $c_{\model,\btheta} \not\equiv c_{\model^\star,\btheta^\star}$ since the posterior probabilities differ significantly in the tails: the Gaussian model underestimates the uncertainty for large values of $|\bx|$. In addition, as illustrated by Figure~\ref{fig:mispeficied}, $c_{\model,\btheta}(\bX)$ does not follow the same distribution under $\{\model,\btheta\}$ and $\{\model^\star,\stheta\}$ implying that $c_{\model,\btheta}$ is not weakly well-specified for soft clustering.

\begin{figure}
    \centering
    \includegraphics[width=0.6\linewidth]{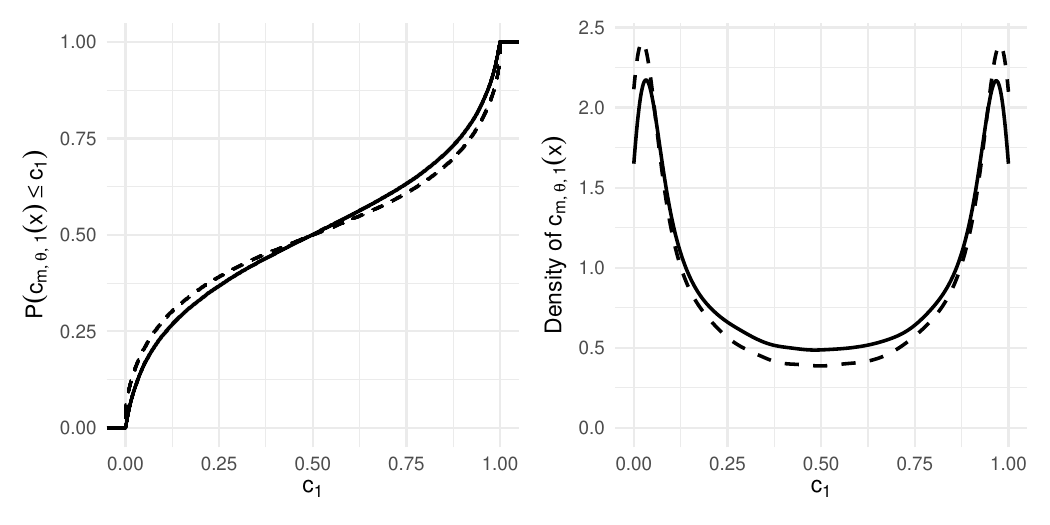}
 \caption{Illustration of weak misspecification for soft clustering. 
    \textbf{Left:}  Cumulative Distribution Function  of the posterior probability $c_{\model,\btheta,1}(\mathbf{X})$. 
    \textbf{Right:} Corresponding density functions. 
    Solid lines represent the distribution of the $c_{\model,\btheta,1}(\mathbf{X})$ under the assumed model $(\model, \btheta)$ (Gaussian mixture model), while dashed lines represent the distribution of $c_{\model,\btheta,1}(\mathbf{X})$ under the true model $(\model^\star, \btheta^\star)$ (Student mixture model).}
    \label{fig:mispeficied}
\end{figure}
\end{example}

While the definition of a weakly well-specified model for soft clustering is formulated using the cumulative distribution function for the sake of generality, our testing procedure is based on moment conditions. We favor this approach because, unlike the CDF which can difficult to compare on the simplex $S^K$ for $K>2$, the moments of the posterior probabilities $c_{\model,\btheta}(\mathbf{X})$ can be effectively evaluated and compared using the empirical likelihood framework, regardless of the number of components.
 Noting that conditional probabilities of classification given $\bX$ are defined in the simplex of size $K$, we consider $\mathcal{E}$ as the set of bounded continuous function defined on $S^K$. Therefore, \eqref{eq:weakcl} is equivalent to have
$$\forall \varphi \in \mathcal{E},\,  \mathbb{E}_{g_{\model,\stheta}}[\varphi( c_{\model,\stheta}(\bX))]=\mathbb{E}_{g_{\model^\star,\btheta^\star}}[\varphi( c_{\model,\stheta}(\bX))],$$
where $\mathbb{E}_{g_{\model,\btheta}}[\varphi( c_{\model,\stheta}(\bX))]=\int_{\mathcal{X}} \varphi( c_{\model,\stheta}(\bx)) g_{\model,\theta}(\bx)d\bx $. 
Let $\psi_{\model,\btheta,\varphi}$ be the function defined for any $\varphi\in \mathcal{E}$ by
$
\psi_{\model,\btheta,\varphi}(\bX) =\varphi(  c_{\model,\stheta}(\bX)) - \mathbb{E}_{g_{\model,\btheta}}[\varphi( c_{\model,\stheta}(\bX)) ]$,
 \eqref{eq:weakcl} is equivalent to have
\begin{equation}\label{eq:nullgr}
\forall \varphi \in \mathcal{E},  \mathbb{E}_{g_{\model^\star,\btheta^\star}}[\psi_{\model,\btheta,\varphi}(\bX)] =0.
\end{equation}

Condition \eqref{eq:nullgr} is precisely what our procedure aims to test.  
Hence, the goodness-of-fit procedure we propose relies on the equality of the functional moments of the conditional probabilities of classification given $\bX$, taken with respect to $g_{\model^\star,\btheta^\star}$ and $g_{\model,\stheta}$.

\subsection{Hypothesis testing for the goodness-of-fit testing procedure of clustering}
If \eqref{eq:nullgr} holds for any $\{\model, \btheta\}$, in practice we consider it with estimator $\wtheta$ that converges in probability to $\stheta$, where $\stheta$ is the parameter minimizing the loss function associated with the clustering step for a given model $\model$.  This loss function is defined as:
$$
\mathcal{L}(\btheta,\model;\btheta^\star,\model^\star) = \mathbb{E}_{g_{\model^\star,\btheta^\star}}[ \zeta(\bX;\btheta,\model,\btheta^\star,\model^\star)],
$$
for some suitable scoring function $\zeta$.  Hence, the estimator $\wtheta$ of $\stheta$ minimizes the empirical loss considered during the clustering step. Considering the random sample $\sample_n = (\bX_1^\top, \ldots, \bX_n^\top)^\top$ composed of $n$ independent copies of $\bX$, where $\bX$ is drawn from $\{\model^\star, \btheta^\star\}$, we have
$$
\wtheta = \argmin_{\btheta \in\Theta_\model} \frac{1}{n} \sum_{i=1}^n\zeta(\bX_i;\btheta,\model,\btheta^\star,\model^\star).
$$
 For instance, when maximum likelihood estimation is conducted, the loss is the Kullback-Leibler divergence between $g_{\model^\star,\btheta^\star}$ and $g_{\model,\btheta}$ defined with $\zeta(\bx;\btheta,\model,\btheta^\star,\model^\star)=\ln g_{\model^\star,\btheta^\star}(\bx) - \ln g_{\model,\btheta}(\bx)$, and $\stheta$ minimizes this function with respect to $\btheta$ in $\Theta_{\model}$. Alternatively, in the case of semi-parametric and non-parametric estimation, the loss function can be the $L_p$ distance (see \citet{HettmanspergerJRSSB2000} for mixtures of symmetric distributions) or the penalized Kullback–Leibler divergence defined using smoothing operators (see \citet{LevineBiometrika2011} for mixtures of univariate densities). To allow the estimation of  $\stheta$  for performing clustering, its uniqueness needs to be assumed, as well as the estimation of the model parameters being conducted according to the specific loss. Hence, this is not an additional requirement introduced by the proposed method for investigating the clustering output but rather an assumption already made during the parameter estimation performed in the clustering step.
{ In our numerical experiments, the estimator $\wtheta$ is typically obtained via the EM algorithm, which consistently minimizes the negative log-likelihood (and thus the Kullback-Leibler divergence) in the parametric case. We observed that the estimators reached the expected $\sqrt{n}$-consistency, and the resulting estimation error remained sufficiently small such that the block-splitting strategy effectively handled the nuisance parameter uncertainty, maintaining a stable Type I error across all scenarios in the parametric case.}

The aim of this paper is to propose a procedure to investigate whether $\{\model,\stheta\}$ is weakly well-specified for soft clustering. 
Hence, based on  \eqref{eq:nullgr} and on the convergence of $\wtheta$ to $\stheta$, the testing procedure considers the null hypothesis
\begin{equation}\label{eq:null}
\mathcal{H}_0:\,\forall \varphi \in \mathcal{E},  \mathbb{E}_{g_{\model^\star,\btheta^\star}}[\psi_{\model,\stheta,\varphi}(\bX)] =0,
\end{equation}
and the alternative hypothesis
$
\mathcal{H}_1:\, \exists \varphi \in \mathcal{E},  \mathbb{E}_{g_{\model^\star,\btheta^\star}}[\psi_{\model,\stheta,\varphi}(\bX)] \neq 0. 
$

For the set of basis functions $\mathcal{E}$, indicator functions or multivariate Bernstein polynomials can be considered. Note that for any $\varphi$, it is easy to compute the empirical counterpart of the moment defined by the null hypothesis from an observed sample composed of independent observations drawn from $\{\model^\star,\btheta^\star\}$. To incorporate the null hypothesis into a testing procedure, two challenges must be addressed. First, the expectation defining the null hypothesis involves an infinite number of functions $\varphi$, whereas only a finite number of moment conditions can be tested in practice. Second, the parameter $\stheta$ is unknown, and we only have access to its estimator $\wtheta$.



 \subsection{Empirical likelihood for goodness-of-fit with an estimator of the  parameters}
 We aim to examine whether $\{\model, \stheta\}$ is weakly well-specified for soft clustering, by using the null hypothesis defined by \eqref{eq:null}. This examination needs to be conducted by considering that $\stheta$ is unknown and that an estimator $\wtheta$ have been computed on the observed sample. Furthermore, we consider $p$ functions from $\mathcal{S}^K$ to $\mathbb{R}$, denoted as $\varphi_{p,1}, \ldots, \varphi_{p,p}$, to construct the $p$-dimensional vector
$\Psi_{\model,p}(\bX;\btheta) = \begin{bmatrix}
    \psi_{\model,\btheta,\varphi_{p,1}}(\bX),
    \ldots, 
    \psi_{\model,\btheta,\varphi_{p,p}}(\bX) 
\end{bmatrix}^\top$.
Under the null hypothesis defined by \eqref{eq:null}, we have
\begin{equation}\label{eq:momentcond}
\mathbb{E}_{g_{\model^\star,\btheta^\star}}[\Psi_{\model,p}(\bX;\stheta) ]=\bzero_p,
\end{equation}
where $\bzero_p$ is the vector of zeros of length $p$.  For any $\btheta$, the empirical likelihood is defined by
$$
L_{\model,p}(\btheta; \sample_n) = \max_{\xi_{\model,p,1},\ldots,\xi_{\model,p,n}}\prod_{i=1}^n \xi_{\model,p,i}(\btheta ),
$$
under the following constraints 
$$\xi_{\model,p,i}(\btheta )\geq 0, \, \sum_{i=1}^n\xi_{\model,p,i}(\btheta )=1 \text{ and } \sum_{i=1}^n \xi_{\model,p,i}(\btheta )\Psi_{\model,p}(\bX_i;\btheta)=\bzero_p,$$ this later being the empirical counterpart of the condition stated by \eqref{eq:momentcond}. 

It is worth noting that since the model parameters  are fixed at this stage of the procedure, the maximization of the empirical likelihood is performed only with respect to the Lagrange multipliers $\boldsymbol{\lambda}$. This leads to a dual optimization problem that is strictly convex. In practice, this optimization is efficiently and stably solved using a Newton-Raphson algorithm, avoiding the numerical instabilities often encountered when empirical likelihood is maximized simultaneously over complex model parameters such as mixture weights and components. Under this dual formulation, we have,

$$\xi_{\model,p,i}(\btheta)^{-1} = n[1 + \lambda_{\model,p}(\btheta)^\top \Psi_{\model,p}(\bX_i;\btheta )],$$ 
where $\lambda_{\model,p}(\btheta)\in\mathbb{R}^{p}$ are the Lagrange multipliers. The empirical log-likelihood ratio is then defined by
$$\mathcal{R}_{\model,p}(\btheta; \sample_n) = \sum_{i=1}^n \ln \left(1 + \lambda_{\model,p}(\btheta )^\top\Psi_{\model,p}(\bX_i;\btheta)\right).$$

Considering a parametric model (\emph{i.e.,} a finite-dimensional parameter space $\Theta_{\model}$) and a fixed number of equations $p$, under mild assumptions, \citet{owen2001empirical} shows that under the null distribution, the statistic $2\mathcal{R}_{\model,p}(\stheta; \sample_n)$ converges to a chi-square distribution with $p$ degrees of freedom. 
Furthermore, if $\Theta_\model \subseteq \mathbb{R}^r$, and if the empirical likelihood $L_{\model,p}(\btheta; \sample_n)$ is maximized over $\btheta \in \Theta_\model$, the clustering model can be tested using the observed sample. In this parametric framework,
\citet[Corollary~4]{QinAOAS1994} shows that, under mild regularity conditions and when
$p>r$, twice the empirical likelihood ratio evaluated at this maximizer converges in
distribution to a chi-square random variable with $p-r$ degrees of freedom. Consequently,
inference on the clustering model can be conveniently carried out using empirical
likelihood methods, provided that the maximization of
$\mathcal{R}_{\model,p}(\btheta; \sample_n)$ over $\Theta_\model$ is tractable. However,
in parametric mixture models, parameter estimation is typically performed by maximizing
the log-likelihood function via the EM algorithm, and the resulting estimator generally
differs from the maximizer of $\mathcal{R}_{\model,p}(\btheta; \sample_n)$. This
discrepancy may render the direct maximization of the empirical likelihood ratio difficult
in practice. Moreover, many clustering procedures are based on semi-parametric mixture
models, in which case the parameter space includes infinite-dimensional components. As a
result, the supremum of $\mathcal{R}_{\model,p}(\btheta; \sample_n)$ over $\btheta$ is no
longer covered by the asymptotic result of \citet[Corollary~4]{QinAOAS1994}.

Finally, considering only a fixed number $p$ of equations may not allow the detection of all alternatives.  Indeed \eqref{eq:null} implies \eqref{eq:momentcond} but there is no equivalence between both equations as illustrated by the following example.
\addtocounter{example}{-1}

\begin{example}[continued]
Since this example considers two components and $c_{\model,\stheta}(\bX)$ lies on the simplex, the functions $\psi_{p,j}$ can only consider the first classification probability $c_{\model,\stheta,1}(\bX)$. Suppose $\psi_{p,j}(\boldsymbol{c}) = \mathds{1}_{c_1 \leq b_{p,j}}$, where the thresholds $b_{p,j}$ partition $[0,1]$ into equal intervals (e.g., $b_{1,1}=1/2$, $b_{2,1}=1/3$, $b_{2,2}=2/3$). 
For $p=1$, the testing procedure cannot reject the null hypothesis because
$
\mathbb{E}_{g_{\model,\stheta}}[\varphi_{1,1}(c_{\model,\stheta}(\bX))] = \mathbb{E}_{g_{\model^\star,\btheta^\star}}[\varphi_{1,1}(c_{\model,\stheta}(\bX))] = 1/2,
$
despite the fact that the model is not weakly well-specified. However, it is able to reject the null hypothesis when $p \geq 2$ since
$
\mathbb{E}_{g_{\model,\stheta}}[\varphi_{p,1}(c_{\model,\stheta}(\bX))] < \mathbb{E}_{g_{\model^\star,\btheta^\star}}[\varphi_{p,1}(c_{\model,\stheta}(\bX))].
$ 
This is a direct consequence of the fact that $\mathbb{P}_{\model,\btheta}(c_{\model,\btheta,1}(\bX) \leq c_1) < \mathbb{P}_{\model^\star,\btheta^\star}(c_{\model,\btheta,1}(\bX) \leq c_1)$ for any $0<c_1 < 1/2$ (see the left part of Figure~\ref{fig:mispeficied}).
\end{example}

Therefore, to allow the detection of all alternatives, it is crucial to let $p$ increase as the sample size $n$ tends to infinity. Empirical likelihood has already been explored in the context of a growing number of equations by \citet{Hjort2009AOS} and \citet{Peng}. By using a chi-square approximation of $2\mathcal{R}_{\model,p}(\stheta; \sample_n)$ and noting that a normalized chi-square random variable with $p$ degrees of freedom converges to a standard Gaussian distribution, these works show that under mild assumptions, $(2\mathcal{R}_{\model,p}(\stheta; \sample_n) - p)/\sqrt{2p}$ converges in distribution to a standard Gaussian distribution under the null hypothesis. However, these results are not directly applicable when $\stheta$ is unknown and replaced by an estimator $\wtheta$ computed from the observed sample $\sample_n$.

We propose a goodness-of-fit procedure to assess whether $g_{\model,\stheta}$ is strongly well-specified for clustering when $\stheta$ is unknown and replaced by an estimator $\wtheta$. 
This procedure is based on $p$ moment equations, where the number of equations increases with the sample size, meaning that $p$ is a function of $n$. It can be seen as an extension of the goodness-of-fit procedure proposed by \citet{braun1980simple}. However, unlike the framework considered in \citet{braun1980simple}, we allow $\btheta$ to have an infinite-dimensional component and consider vectors with an increasing dimension. This testing procedure begins by splitting the original sample $\sample_n$ into $B$ sub-samples $\sample^{(1)},\ldots, \sample^{(B)}$, such that each observation is assigned to exactly one sub-sample. Each sub-sample $\sample^{(b)}$ contains $n_b$ observations. The sizes $n_1,\ldots,n_{B}$ of the sub-samples $\sample^{(1)},\ldots, \sample^{(B)}$ and the number of sub-samples $B$ increase with the sample size at rates specified in the next section, making them functions of $n$. For each sub-sample $\sample^{(b)}$, with $1\leq b\leq B$, we compute the statistic $Y_{\est}$, defined as twice the empirical likelihood ratio evaluated at $\btheta=\wtheta$ and computed on sample $\sample^{(b)}$, such that for any $\btheta$, we have
\begin{equation}\label{Sec2: Yany}
Y_{\any}=2\mathcal{R}_{\model, p}(\btheta; \sample^{(b)})    .
\end{equation}

Finally, the expectation $\mathbb{E}_{g_{\model,\btheta}}[\varphi( c_{\model,\btheta}(\bX)) ]$ is generally not explicit, since the distribution of the posterior classification probabilities is unknown in closed form, as the mapping $\bx \mapsto c_{\model,\btheta}(\bx)$ is rarely one-to-one. However, it can be derived in simple cases as shown by the following example. When the the expectation $\mathbb{E}_{g_{\model,\btheta}}[\varphi( c_{\model,\btheta}(\bX)) ]$ is not explicit, it can be approximated numerically using Monte Carlo simulations, since it is easy to generate observations from a mixture model. Note that the accuracy of the approximation only depends on the number of simulations, and thus the user can choose a sufficiently large number of replications to make the approximation error negligible as a consequence of Theorem~\ref{thm:alternative}. 

\begin{example}[Two-component homoscedastic multivariate Gaussian mixture]
Consider $\bX \in \mathbb{R}^d$ following the mixture density $g_{\model,\btheta}(\bx) = \pi_1 \phi(\bx; \boldsymbol{\mu}_1, \boldsymbol{\Sigma}) + \pi_2 \phi(\bx; \boldsymbol{\mu}_2, \boldsymbol{\Sigma})$, where $\boldsymbol{\Sigma}$ is a common positive definite covariance matrix. The first posterior probability $c_{\model,\btheta,1}(\bx)$ can be written in the form of a logistic function:
$c_{\model,\btheta,1}(\bx) = 1/(1 + \exp(-(\ba^\top \bx + b)))$, 
where $\ba = \boldsymbol{\Sigma}^{-1}(\boldsymbol{\mu}_1 - \boldsymbol{\mu}_2)$ and $b = \ln(\pi_1/\pi_2) - \frac{1}{2}\boldsymbol{\mu}_1^\top \boldsymbol{\Sigma}^{-1} \boldsymbol{\mu}_1 + \frac{1}{2}\boldsymbol{\mu}_2^\top \boldsymbol{\Sigma}^{-1} \boldsymbol{\mu}_2$. Let $W = \ba^\top \bX + b$ be the discriminant score. Since $\bX$ follows a mixture of two Gaussians, $W$ follows a univariate Gaussian mixture $\pi_1 \mathcal{N}(m_1, s^2) + \pi_2 \mathcal{N}(m_2, s^2)$ with $m_k = \ba^\top \boldsymbol{\mu}_k + b$ and $s^2 = \ba^\top \boldsymbol{\Sigma} \ba = (\boldsymbol{\mu}_1 - \boldsymbol{\mu}_2)^\top \boldsymbol{\Sigma}^{-1} (\boldsymbol{\mu}_1 - \boldsymbol{\mu}_2)$. The cumulative distribution function of $c_{\model,\btheta,1}(\bX)$ for $c_1 \in (0,1)$ is then explicitly given by:
$
\mathbb{P}(c_{\model,\btheta,1}(\bX) \leq c_1) = \pi_1 \Phi\left( \frac{\text{logit}(c_1) - m_1}{s} \right) + \pi_2 \Phi\left( \frac{\text{logit}(c_1) - m_2}{s} \right)$, 
where $\text{logit}(u) = \ln(u/(1-u))$ and $\Phi$ is the CDF of the standard normal distribution. While this explicit form exists for this fundamental case, it becomes analytically intractable as soon as the covariance matrices differ or the number of components increases. Our procedure addresses this by using Monte Carlo simulations to approximate this distribution for any arbitrary mixture model.
\end{example}

\section{Test statistics obtained from the empirical likelihood ratio tests} \label{sec:theo} 
 
\subsection{Controlling the asymptotic distribution of the empirical likelihood ratio of one block}

In this section, we show that each of the $B$ random variables $Y_{\est}$ can be treated as arising from independent chi-squared variables with $p$ degrees of freedom, under the null hypothesis.  This distributional result serves as the foundation for the proposed test statistics used to validate the clustering methodology. To prove this result, stated in Theorem~\ref{thm:niveau}, we require some assumptions, which we now briefly discuss before stating them formally. 

Ensuring the convergence in distribution of any empirical likelihood ratio requires assumptions on the covariance matrix
 $$\bSigma_{\model,p}= \mathbb{E}[\Psi_{\model,p}(\bX;\stheta)\Psi_{\model,p}(\bX;\stheta)^\top].$$
 When $p$ is fixed, the usual assumption  is that the singular values of $\bSigma_{\model,p}$  are strictly positive (see the assumptions in \citet[Lemma~1]{QinAOAS1994}). In this paper, since $p$ increases with the sample size, we consider the following extension of this assumption, already introduced in \citet[condition (D6)]{Hjort2009AOS} and described in Assumption~\ref{ass:main}-\ref{ass:cov}. 
 
The other assumptions are required to account for a growing number of equations. Indeed, the null hypothesis stated in \eqref{eq:null} considers an infinite number of functions $\varphi$, whereas the moment conditions defined in \eqref{eq:momentcond} consider $p$ functions. Therefore, we impose that $p$ tends to infinity as $n$ tends to infinity in order to be able to detect any alternative. In addition, ensuring the convergence of the empirical covariance matrix to $\bSigma_{\model,p}$ is achieved by controlling the $q_0$-th order moment of $\psi_{\model,\wtheta,\varphi_{p,j}}(\bX)$ for any $(p,j)$, with $q_0\geq 4$. Since $\stheta$ is unknown and is replaced by $\wtheta$, the impact of replacing $\Psi_{\model,p}(\bx;\stheta)$ with $\Psi_{\model,p}(\bx;\wtheta)$ must be controlled. Such control has been established in \citet[Section A.2]{ChaumarayAOS} uniformly on $\bx$ in the case of a semi-parametric regression model. Obviously, this control depends on the rate of convergence, in probability, of $\wtheta$ to $\stheta$. This rate depends on the estimation procedure performed during the clustering step and cannot be improved. In \cite{butucea2014semiparametric}, the authors constructed a deconvolution estimator of $\stheta$ and give a pointwise rate of convergence. To the best of our knowledge, uniform convergence results are not yet available in the nonparametric setting. Since our block construction depends on this convergence rate, we formulate the methodology in terms of a generic rate.
This allows the proposed procedure and theoretical results to apply in a general framework, covering both parametric and nonparametric models,  Nonetheless, the nonparametric case requires specific convergence rate results which, at present, remain unavailable in the literature.  Therefore, we adapt the procedure according to this rate of convergence by making assumptions on the growth of the size of each subsample and the number of subsamples. In particular, in Assumption~\ref{ass:main}-\ref{ass:growing1} we assume that the size of each block tends to infinity at the same rate of order  $n^\rho$, leading to the number of blocks $B$ tending to infinity at a rate of order $n^{1-\rho}$. This requirement has already been made by \citet{braun1980simple} to perform a goodness-of-fit procedure based on Kolmogorov-Smirnov or Cramér-von Mises statistics. Note that this assumption is not restrictive, as it only imposes conditions on the growth of $n_b$'s and $B$. However, satisfying this assumption requires knowledge of the rate of convergence of the estimator $\wtheta$. All these conditions are stated in Assumption~\ref{ass:main}.

\begin{assumptions} \label{ass:main} 
\leavevmode\par
\begin{enumerate}
\item \label{ass:cov} For any $p$, all singular-values of  $\bSigma_{\model,p}$ are upper bounded by $\sigma$ and lower-bounded by $\varsigma$ with $\varsigma>0$ and $\sigma<\infty$.
    \item \label{ass:cvtheta1} $\stheta$ is the unique minimizer of $\mathcal{L}(\btheta,\model;\btheta^\star,\model^\star)$ with respect to $\btheta\in\Theta_{\model}$.
\item \label{ass:cvtheta2} There exists $\tau$ such that $\tau>1/3$ and for any $(p,j)$, $\max_{1\leq i \leq n} | \psi_{\model,\wtheta,\varphi_{p,j}}(\bX_i) -  \psi_{\model,\stheta,\varphi_{p,j}}(\bX_i)|=O_\mathbb{P}(n^{-\tau})$,
\item \label{ass:growing1} There exists $\rho$ with $2/3<\rho<2\tau$ such that for any $1\leq b\leq B$, $\lim_{n\to\infty} n_b n^{-\rho}=1$ 
    \item \label{ass:growing2}  There exists an integer $q_0\geq 4$, and two positive reals $r_0$ and $\tilde C$ such that for any $(p,j)$, we have $\mathbb{E}_{g_{\model^\star,\theta^\star}} | \psi_{\model,\btheta,\varphi_{p,j}}(\bX)|^{q_0}$ is upper-bounded by $\tilde C p^{r_0}$.
    \item \label{ass:growing3} There exists $0<\kappa< (\rho/6 - 1/6q_0)/(1+r_0/q_0)$ such that $p=[n^{\kappa}]$
 \end{enumerate}
\end{assumptions}

\begin{remark} : Before proceeding further, we discuss some implications of the assumptions stated above. We then include a section on practical implementation to guide the reader in choosing suitable basis functions and other components of our procedure so that the assumptions are satisfied and Theorems \ref{thm:niveau} and \ref{thm:alternative} remain valid.
\begin{itemize}
    \item From Assumption~\ref{ass:main}-\ref{ass:cvtheta2}, we have
$
\sup_{1\leq i\leq n} \|\Psi_{\model,p}(\bX_i;\stheta) -  \Psi_{\model,p}(\bX_i;\wtheta)\|_2=O_\mathbb{P}(n^{-\tau}p^{1/2})$. 

\item From Assumption~\ref{ass:main}-\ref{ass:cvtheta2} up to \ref{ass:main}-\ref{ass:growing3}, the number of equations $p$ tends to infinity as $n_b$ tends to infinity at a rate that satisfies $p^{6}n_b^{-1} = o(1)$ and $n^{-\tau}p^{6}=o(1)$. 

\item Assumption~\ref{ass:main}-\ref{ass:growing2} and Jensen inequality applied to the function $u\mapsto u^2$ imply that
$
\|\Psi_{\model,p}(\bX_i^{(b)};\stheta)\|_2^{q_0} \leq p^{q_0/2-1} \sum_{j=1}^p |\psi_{\model,\stheta,\varphi_{p,j}}(\bX_i^{(b)})|^{q_0}$. Hence, we have
\begin{equation} \label{eq:moment}
  \mathbb{E}[\|\Psi_{\model,p}(\bX_i^{(b)};\stheta)\|_2^{q_0}] \leq\tilde{C} p^{q_0/2+r_0}.
\end{equation}
\item If more restrictive conditions are assumed, that is Assumption~\ref{ass:main}-\ref{ass:growing2} holds for more moments, meaning that the growth condition is close to the case of bounded variables then the rate of $p$ is as large and close to $n^{\rho/6}$. This rate is slower than the one obtained in \cite{Hjort2009AOS}, that is $p=o(n^{1/3})$ when $\rho=1$. However, the study of the empirical likelihood with growing dimension conducted in this later does not incorporate the estimation of the parameter $\stheta$ and therefore does not have to handle the negligibility of this error, nor the management of the blocks to account for it.

\item The basis functions considered in this paper allow, by their properties, Assumption~\ref{ass:main}-\ref{ass:cvtheta2} to be verified if for instance the function $c_{\model,\btheta}(\bx)$ is Lipschitz w.r.t $\btheta$ uniformly in $\bx$ which happens in the parametric case as long as  $c_{\model,\btheta}(\bx)$ is continuously differentiable in $\btheta$, and there exists a constant $M>0$ such that the norm of the gradient $\nabla_{\btheta}c_{\model,\btheta}(\bx)$ is bounded by $M$ uniformly over $\bx$. 

\item 
Note that Assumptions~\ref{ass:main}-\ref{ass:cov} up to \ref{ass:main}-\ref{ass:cvtheta2}, and \ref{ass:main}-\ref{ass:growing2} regarding the basis functions $\psi$ and the model regularity are verifiable by construction, as they depend on the user's choice of basis (\emph{e.g.,} Bernstein polynomials) and the use of standard mixture families. While the asymptotic growth rates in Assumptions~\ref{ass:main}-\ref{ass:growing1} and  \ref{ass:main}-\ref{ass:growing3} cannot be directly checked on a finite sample, their practical impact is extensively addressed through the sensitivity analysis in Section 5.1. Our results show that the procedure is robust to mild deviations in parameter tuning, provided the number of observations per block $n_b$ is sufficient to ensure stable parameter estimation.

\end{itemize}
\end{remark}

\begin{theorem}\label{thm:niveau} 
If Assumptions~\ref{ass:main} hold true then, under the null hypothesis  stated by \eqref{eq:null}, the supremum of difference cumulative distribution functions of  the $B$-dimensional vector $Y_{\estall}=(Y_{\estall, 1},\ldots,Y_{\estall , B})^\top$ and the  cumulative distribution functions of a vector of $B$ independent chisquare random variables with $p$ degrees of freedom tends to zero meaning that
$$
\lim_{n\to \infty}   \left\| F_{Y_{\estall}} - \prod_{b=1}^B F_{\mathcal{X}^2_p}\right\|_\infty = 0,
$$
where $F_{Y_{\estall}} $ and $F_{\mathcal{X}^2_p}$ denote  the cumulative distribution function of the $B$-dimensional vector $Y_{\estall}$ and of a chi-square random variable with $p$ degrees of freedom respectively.
\end{theorem}

We provide a sketch of the proof of Theorem~\ref{thm:niveau}. The full proof is postponed in Appendix~\ref{sec:appproof}. 
\begin{proof}[Sketch of Proof of Theorem~\ref{thm:niveau}]
The first part consists in generalizing the results stated in \cite{owen2001empirical} to the case of growing dimension and nuisance parameters and providing more accurate stochastic orders of the Taylor remainder terms in order to be able to show that
\begin{equation}\label{Sec3:eq1}
\left(Y_{\estall}-W_{\trueb}\right)^{\dagger}=o_\mathbb{P}(1)    
\end{equation}
where $Y_{\est}$  is defined in \eqref{Sec2: Yany} and $W_{\true}$ 
given by
\begin{equation}\label{Sec3: W}
W_{\true}=\bZ_{\true}^\top\bSigma_{\model,p}^{-1}  \bZ_{\true},  
\end{equation}
where \begin{equation}\label{Sec3:Zany}
\bZany =n_b^{-1/2} \sum_{i=1}^{n_b}  \Psi_{\model,p}(\bX_i^{(b)};\btheta  ).
\end{equation}A convergence in distribution of $(2p)^{-1/2} (W_{\true} -p)$ towards the standard gaussian distribution is already established in \cite{Peng}. However, this convergence is not easily manageable, as we have to deal with the maximum over $B$ statistics to circumvent the issue arising from the use of an estimator of the model parameters. Under Assumptions~\ref{ass:main} and technical lemmas proved in Section \ref{B.2}, \eqref{Sec3:eq1} is established.

The second part of the proof consists in showing that $W_{\trueb}=(W_{\trueb,1},\ldots,W_{\trueb,B})^\top$ behaves like a $B$ dimensional vector composed of $B$ independent chi-square random variables with $p$ degrees of freedom, when $n$ is large. This last point is proved under a Berry-Esseen bound \citep{bentkus2003dependence}.

\end{proof}

\subsection{Test statistics obtained from the asymptotic distribution of the empirical likelihood ratio of one block} \label{subsec:allstat}

Based on the results established in Theorem~\ref{thm:niveau}, we achieve asymptotic control over the distribution of the $B$-dimensional vector $Y_{\estall}$. Specifically, under the null hypothesis $\mathcal{H}_0$, the components of $Y_{\estall}$ are asymptotically independent and identically distributed according to a chi-square distribution with $p$ degrees of freedom. This theoretical guarantee allows us to construct global goodness-of-fit tests by considering various functional transformations of this vector.

To assess the relevance of the model-based clustering, we first consider two aggregate measures: the sum statistic, which captures global misspecifications across all blocks, and the maximum statistic, which is specifically sensitive to local discrepancies. However, these summary statistics may overlook subtle but systematic deviations in the distribution of the log-likelihood ratios. Therefore, we also propose leveraging the full information contained in $Y_{\estall}$ by performing goodness-of-fit (GOF) tests that compare the empirical distribution of the sample $\{Y_{\estall, b}\}_{b=1}^{B}$ to the theoretical $\chi^2_p$ law. Among these, the Kolmogorov-Smirnov, Cramér-von Mises, and Anderson-Darling tests are particularly suitable for evaluating the entire shape of the distribution. In our numerical experiments, this strategy of performing a GOF test on the components of $Y_{\estall}$ consistently yielded the best results in terms of power and level control.

The five considered statistics are the sum statistic, the maximum statistic, the Kolmogorov-Smirnov statistic, the Cramér-von Mises statistic and the Anderson-Darling statistic defined respectively by:
$$\bar{Y}_{\estall} = \sum_{b=1}^{B} Y_{\estall, b},\, 
  Y_{\estall}^{\dagger} = \max_{1 \le b \le B} Y_{\estall, b},$$ 
 $$\texttt{KS}[Y_{\estall}] = \sup_{y} |\hat{F}_{B}(y) - F_{\chi^2_p}(y)|,\,
 \texttt{CvM}[Y_{\estall}] = B \int_{-\infty}^{\infty} (\hat{F}_B(y) - F_{\chi^2_p}(y))^2 dF_{\chi^2_p}(y) $$
 and
 $$
\texttt{AD}[Y_{\estall}] = B \int_{-\infty}^{\infty} \frac{(\hat{F}_B(y) - F_{\chi^2_p}(y))^2}{F_{\chi^2_p}(y)(1 - F_{\chi^2_p}(y))} dF_{\chi^2_p}(y),$$
where $\hat{F}_{B}(y) = \frac{1}{B} \sum_{b=1}^{B} \mathds{1}_{\{Y_{\estall, b} \leq y\}}$ is the empirical cumulative distribution function and $F_{\chi^2_p}$ is the cumulative distribution function of a chi-square distribution with $p$ degrees of freedom.
The following corollary formalizes the asymptotic distributions and defines the rejection regions at a significance level $\alpha$ for each statistic. 
\begin{corollary}\label{cor:gof_stats}
Under the null hypothesis $\mathcal{H}_0$ and the assumptions of Theorem~\ref{thm:niveau}, as $n \to \infty$, the rejection regions at level $\alpha$ are defined as follows:
\begin{itemize}
    \item For the sum statistic:
    $\mathcal{K}_{\alpha}^{\text{sum}} = \left\{ \bar{Y}_{\estall} > q_{1-\alpha}(\chi^2_{pB}) \right\}$, 
    where $q_{1-\alpha}(\chi^2_{p B})$ is the $(1-\alpha)$-quantile of the chi-square distribution with $p B$ degrees of freedom.

    \item For the maximum statistic:
    $\mathcal{K}_{\alpha}^{\max} = \left\{ Y_{\estall}^{\dagger} > q_{(1-\alpha_B)}(\chi^2_{p}) \right\}$, 
    where $q_{1-\alpha}(\chi^2_{p})$ is the $(1-\alpha)$-quantile of the chi-square distribution with $p$ degrees of freedom and $\alpha_B=1-(1-\alpha)^{1/B}$.

    \item For the Kolmogorov-Smirnov statistic:
    $\mathcal{K}_{\alpha}^{KS} = \left\{ \sqrt{B} \texttt{KS}[Y_{\estall}]  > K_{\alpha} \right\}$, 
    where $K_{\alpha}$ is the $(1-\alpha)$-quantile of the Kolmogorov distribution.

    \item For the Cramér-von Mises statistic:
    $\mathcal{K}_{\alpha}^{CvM} = \left\{ \texttt{CvM}[Y_{\estall}] > cvm_{\alpha} \right\}$, 
    where $cvm_{\alpha}$ is the $(1-\alpha)$-quantile of the Cramér-von Mises distribution.

    \item For the Anderson-Darling statistic:
    $\mathcal{K}_{\alpha}^{AD} = \left\{ \texttt{AD}[Y_{\estall}] > ad_{\alpha} \right\}$, 
    where $ad_{\alpha}$ is the $(1-\alpha)$-quantile of the Anderson-Darling distribution.
\end{itemize}
\end{corollary}

\subsection{Local alternative}
To study the power properties of the testing procedure, we consider a local alternative defined as a perturbation of the moment conditions. This approach is particularly relevant in the context of model-based clustering for two reasons. First, it allows us to account for the potential bias introduced by the numerical approximation of the expectation $\mathbb{E}_{g_{\model,\btheta}}[\varphi( c_{\model,\btheta}(\bX)) ]$, which is generally not explicit and must be estimated via Monte Carlo methods. Second, it characterizes the sensitivity of the test to local misspecifications of the modeling of the posterior probabilities of classification. Under this framework, we consider a perturbation of the moment condition defined, for any $\Upsilon\in\mathbb{R}^p$ with $\|\Upsilon\|_2=1$, by:
$$
\Psi_{\model,p}^{\alt}(\bX_i^{(b)};\stheta)=\Psi_{\model,p}(\bX_i^{(b)};\stheta)+  \Upsilon n_b^{-1/2}.
$$
The empirical log-likelihood ratio is then defined by
$$\mathcal{R}_{\model,p}^{\alt}(\btheta; \sample_n) = \sum_{i=1}^n \ln \left(1 + \lambda_{\model,p}^{\alt}(\btheta )^\top\Psi_{\model,p}^{\alt}(\bX_i;\btheta)\right),$$
where the Lagrange multipliers $\lambda_{\model,p}^{\alt}(\btheta )$ denote the multipliers associated with the perturbed constraints. Hence, we consider the $B$-dimensional vector $Y_{\estall}^{\alt}=(Y_{\estall, 1}^{\alt},\ldots,Y_{\estall , B}^{\alt})^\top$ statistic obtained with a local perturbation of the moment condition where
$$
Y_{\any}^{\alt}=2\mathcal{R}_{\model, p}^{\alt}(\btheta; \sample^{(b)}) .
$$

By shifting the moment conditions by an order of $n_b^{-1/2}$, we examine the test's ability to detect subtle deviations in the conditional distribution of the latent variables. Since the parameters are estimated beforehand, this local alternative specifically targets the validity of the posterior classification laws, aiming to identify if the selected model $\model$ remains consistent with the true underlying clustering structure even under small distributional shifts.

The following theorem establishes the asymptotic distribution of the test statistic under such local alternatives:

\begin{theorem}\label{thm:alternative} 
If Assumptions~\ref{ass:main} hold true then, under the null hypothesis  stated by \eqref{eq:null}, the supremum of difference cumulative distribution functions of $Y_{\estall}^{\alt}$ and the  cumulative distribution functions of a vector of $B$ independent non-central chi-square random variables with $p$ degrees of freedom tends to zero meaning that
$$
\lim_{n\to \infty} \| F_{Y_{\estall}^{\alt}} - \prod_{b=1}^{B} F_{\chi^2_{p}(\Upsilon^\top\bSigma_{\model,p}^{-1}  \Upsilon)}\|_\infty = 0,
$$
where $F_{Y_{\estall}^{\alt}} $ and $F_{\chi^2_{p}(\Upsilon^\top\bSigma_{\model,p}^{-1}  \Upsilon)}$ denote  the cumulative distribution function of the $B$-dimensional vector $Y_{\estall}^{\alt}$ and a non-central chi-square random variable with $p$ degrees of freedom and non-centrality parameter equal to $\Upsilon^\top\bSigma_{\model,p}^{-1}  \Upsilon$ respectively.
\end{theorem}

 \section{Practical implementation}\label{sec:prac:imp}

The proposed testing procedure requires the calibration of several tuning parameters: the choice of the test statistic, the functional basis, the number of moment conditions $p$, the number of blocks $B$, and the number of Monte Carlo replications $M$. This section provides practical guidelines for these choices based on our theoretical results and numerical findings.

\subsection*{Choice of the statistics}
While several statistics were introduced in Section~\ref{subsec:allstat}, we recommend using the Kolmogorov-Smirnov (KS) or Cram\'er-von Mises (CvM) types. As shown in Section~\ref{sec:num}, these statistics exhibit superior numerical performance in small samples ($n=200$) compared to the \textit{max} or \textit{sum} statistics. By considering the full distribution of the moment vector rather than its individual components, they provide a more robust diagnostic of the posterior probability distribution.

\subsection*{Choice of the basis}
Theorem~\ref{thm:niveau} requires basis functions that satisfy Assumptions~\ref{ass:main}.\ref{ass:cov} and~\ref{ass:main}.\ref{ass:growing2}. While indicator functions allow for a theoretical construction satisfying the identity covariance matrix assumption through PCA-based orthogonalization, defining the partition sets on the simplex is challenging for $K>2$. Consequently, we recommend using a Bernstein basis. To satisfy Assumption~\ref{ass:main}.\ref{ass:cov} and handle the potential multicollinearity on the simplex, we propose the following data-driven construction:
\begin{itemize}
    \item Consider $K$-variate Bernstein polynomials of increasing degrees $s=1, 2, \dots$ until the number of polynomials exceeds $p$. Recall that a $K$-variate Bernstein polynomial of degree $s$ is defined by $\varpi_{j_1,\ldots,j_K}(a_1,\ldots,a_K) = \frac{s!}{\prod_{k=1}^K j_k!}a_k^{j_k}$ where $\sum_{k=1}^K j_k = s$.
    \item Perform a Principal Component Analysis (PCA) on these polynomials evaluated on the data (or simulated data under $H_0$) to ensure the invertibility of the covariance matrix $\Sigma_{m,p}$.
    \item The final functions used in the moment conditions \eqref{eq:null} are composed of the first $p$ principal components that are associated with an eigen values greater than a specific threshold, say $10^{-4}$. This threshold aim to ensure that the smallest eigen value of the covariance matrix is bounded away from zero.
\end{itemize}
Our experiments demonstrate that this approach offers high stability and handles the simplex structure naturally by removing redundant information.

\subsection*{Number of blocks and moment conditions}
The number of blocks $B$ primarily impacts the \textit{max} and \textit{sum} statistics. For KS and CvM, the procedure is remarkably robust to the choice of $B$. Regarding the number of moment conditions $p$, it must increase with $n$ to ensure the consistency of the test against any alternative. Empirically, a slow growth rate is sufficient to maintain the nominal level while providing high power.

\subsection*{Monte Carlo approximation}
The target moments $\mu_{m,\theta}$ in \eqref{eq:momentcond} are generally not available in closed form and must be approximated by $M$ Monte Carlo replications. This requires only the ability to simulate data from the estimated mixture model. According to Theorem~\ref{thm:alternative}, $M$ must grow fast enough so that the simulation error becomes negligible compared to the sampling variance. Empirically, the test fails to control the Type I error if $M$ is too small (e.g., $\xi=1/2$ in Table~\ref{tab:H0Monte}), as the numerical noise contaminates the statistic.

\subsection*{Recommended default values}
Based on our simulation study, considering a parametric mixture model, we suggest the following default settings for the Bernstein basis: $B = \lfloor 2n^{1/4} \rceil, \quad p = \lfloor 3 n^{1/9} \rceil, \quad \text{and} \quad M = \lfloor 10 n^{3/4} \rceil$. 
These values ensure that the requirements of Assumptions~\ref{ass:main} and Theorem 2 are satisfied for a wide range of data types.  The proposed method is obtained by the R package GOFclustering available on CRAN and all the numerical experiments are conducted with this R package.

\section{Numerical experiments} \label{sec:num} 

\subsection{Impact of the tuning parameters on the test performance under the null hypothesis}
In this section, we investigate the sensitivity of the proposed testing procedure to various tuning parameters under the null hypothesis. Specifically, we compare the performance of the different test statistics introduced in Section~\ref{subsec:allstat} and evaluate the impact of choosing different functional bases (\emph{e.g.,} Bernstein polynomials versus indicator functions). Furthermore, we analyze how the structural configuration of the test affects the results, focusing on the number of blocks $B$ and the dimension $p$ of the moment vector. Finally, we examine the influence of the number of Monte Carlo replications used to numerically approximate the moment condition, ensuring the stability and robustness of our approach.

In this section, data are generated from a bi-component Gaussian mixture model with equal proportions and identity covariance matrices. The observations are described by $d$ variables, and the centers of the components are defined as $\mu_1=(1,\ldots,1)^\top/\sqrt{d}$ and $\mu_2=(-1,\ldots,-1)^\top/\sqrt{d}$. Clustering is then performed using a Gaussian mixture model with diagonal covariance matrices, estimated via an EM algorithm implemented in the R package VarSelLCM \citep{bioinfo}. To ensure the reliability of the estimator $\widehat{\theta}_{m,n}$, the EM algorithm is systematically run with multiple random initializations (multi-start strategy). This approach minimizes the risk of converging to poor local maxima, thereby ensuring that the estimation error remains within the bounds required by Assumption 1.3 and does not bias the test's level in finite samples.

\subsubsection*{Impact of test statistics and functional bases}
Table~\ref{tab:H0stat} indicates the empirical rejection rates obtained with different test statistics and functional bases on 5000 samples by the proposed method used with $B=\lfloor 2n^{1/4} \rceil$, $p=\lfloor 3 n^{1/9} \rceil$,  $M=\lfloor 10 n^{3/4} \rceil$  and a nominal level of $\alpha=0.05$.
\begin{table}[ht]
\centering
\begin{tabular}{cccc|ccccc|ccccc}
  \hline
&& & & \multicolumn{5}{c|}{Bernstein basis}& \multicolumn{5}{c}{Indicator basis} \\
n & $B$ &$p$ & $M$ & max & sum & KS & CvM & AD & max & sum & KS & CvM & AD \\ 
  \hline
200 & 6 & 5 & 531 & 0.248 & 0.225 & 0.069 & 0.072 & 0.182 & 0.936 & 0.936 & 0.317 & 0.281 & 0.940 \\ 
  400 & 7 & 6 & 894 & 0.059 & 0.059 & 0.052 & 0.050 & 0.052 & 0.552 & 0.552 & 0.082 & 0.083 & 0.558 \\ 
  800 & 8 & 6 & 1504 & 0.049 & 0.046 & 0.051 & 0.048 & 0.048 & 0.117 & 0.134 & 0.053 & 0.058 & 0.119 \\ 
  1600 & 9 & 7 & 2529 & 0.048 & 0.046 & 0.050 & 0.048 & 0.047 & 0.072 & 0.096 & 0.064 & 0.064 & 0.071 \\ 
  3200 & 10 & 7 & 4254 & 0.048 & 0.052 & 0.048 & 0.049 & 0.049 & 0.070 & 0.100 & 0.066 & 0.064 & 0.065 \\ 
  6400 & 12 & 8 & 7155 & 0.059 & 0.051 & 0.050 & 0.050 & 0.049 & 0.069 & 0.095 & 0.053 & 0.055 & 0.061 \\ 
   \hline
\end{tabular}
\caption{Empirical level of the different test statistics under $H_0$ using Bernstein and indicator bases across various sample sizes $n$ ($\alpha=0.05$).}
\label{tab:H0stat}
\end{table}
Table~\ref{tab:H0stat} shows that all five statistics achieve the nominal level as the sample size $n$ increases, for both the Bernstein and indicator bases. However, the Bernstein basis appears more stable for smaller sample sizes and offers a significant practical advantage as it does not require the estimation of boundaries, making it easier to implement. Regarding the test statistics, those utilizing the full distribution (KS, CvM, and AD) outperform, for small samples, the max and sum statistics, which are based on specific components of the moment vector. In light of these results, and given their consistent performance, we will henceforth focus on the Bernstein basis combined with the KS and CvM statistics for the remainder of this study. 

To evaluate the stability of the test's conclusion with respect to the random split-sample construction, we conducted a reproducibility experiment.  For each statistic, we define the  consensus decision as the majority vote (reject or fail to reject $H_0$) across these 50 repetitions. We report in Table~\ref{tab:Repeat} the mean proportion of samples where the majority decision was  fail to reject $H_0$  (correct decision under $H_0$), along with the standard deviation of this majority outcome across the 1,000 datasets.
\begin{table}[ht]
\centering
\small
\begin{tabular}{cccc|cc|cc|cc|cc|cc}
  \hline
& & & & \multicolumn{2}{c|}{Max} & \multicolumn{2}{c|}{Sum} & \multicolumn{2}{c|}{KS} & \multicolumn{2}{c|}{CvM} & \multicolumn{2}{c}{AD} \\
$n$ & $B$ & $p$ & $M$ & mean & sd & mean & sd & mean & sd & mean & sd & mean & sd \\  
  \hline
200 & 6 & 5 & 531 & 0.793 & 0.101 & 0.815 & 0.099 & 0.938 & 0.047 & 0.935 & 0.049 & 0.854 & 0.093 \\ 
400 & 7 & 6 & 894 & 0.938 & 0.038 & 0.943 & 0.036 & 0.948 & 0.032 & 0.948 & 0.032 & 0.945 & 0.033 \\ 
800 & 8 & 6 & 1504 & 0.951 & 0.030 & 0.953 & 0.031 & 0.950 & 0.031 & 0.950 & 0.030 & 0.950 & 0.030 \\ 
1600 & 9 & 7 & 2529 & 0.950 & 0.031 & 0.952 & 0.031 & 0.950 & 0.031 & 0.950 & 0.031 & 0.950 & 0.031 \\ 
3200 & 10 & 7 & 4254 & 0.949 & 0.031 & 0.950 & 0.031 & 0.950 & 0.031 & 0.949 & 0.031 & 0.949 & 0.031 \\ 
6400 & 12 & 8 & 7155 & 0.948 & 0.032 & 0.946 & 0.033 & 0.949 & 0.032 & 0.949 & 0.031 & 0.949 & 0.031 \\ 
   \hline
\end{tabular}
\caption{Stability of the majority decision across 50 random splits under $H_0$ (Bernstein basis). The mean represents the rate of consistent non-rejection across the datasets, and sd reflects the variability of this majority consensus.}
\label{tab:Repeat}
\end{table}

As shown in Table~\ref{tab:Repeat}, the consensus decision is remarkably stable. For small samples ($n=200$), the KS and CvM statistics provide a very reliable majority vote, with a mean of 0.938 and a low standard deviation (0.047), indicating that the random split rarely alters the final conclusion. For $n \geq 400$, all statistics achieve a stability of approximately 0.95 with a minimal standard deviation, proving that the test results are not sensitive to the specific random split chosen by the user.

\subsubsection*{Impact of the number of blocks}
Table~\ref{tab:H0B} reports the empirical rejection rates obtained for different growth rates of the number of blocks, based on 5000 samples, using the proposed method with $B=\lfloor 2n^{1-\rho} \rceil$, $p=\lfloor 3 n^{1/9} \rceil$,  $M=\lfloor 10 n^{3/4} \rceil$ and a nominal level of $\alpha=0.05$.
\begin{table}[ht]
\centering
\begin{tabular}{ccc|cc|cc|cc|cc|cc}
  \hline
  & &$\rho$& \multicolumn{2}{c|}{$6/7$}& \multicolumn{2}{c|}{$5/6$}& \multicolumn{2}{c|}{$4/5$}& \multicolumn{2}{c|}{$3/4$}& \multicolumn{2}{c}{$1/2$} \\
  $n$ & $p$ &$M$ & KS & CvM & KS & CvM & KS & CvM & KS & CvM & KS & CvM \\ 
  \hline
200 & 5 & 531 & 0.059 & 0.057 & 0.056 & 0.060 & 0.056 & 0.063 & 0.073 & 0.074 & 0.991 & 0.983 \\ 
  400 & 6 & 894 & 0.056 & 0.058 & 0.047 & 0.044 & 0.051 & 0.052 & 0.050 & 0.049 & 0.941 & 0.886 \\ 
  800 & 6 & 1504 & 0.055 & 0.051 & 0.053 & 0.053 & 0.048 & 0.053 & 0.049 & 0.051 & 0.477 & 0.489 \\ 
  1600 & 7 & 2529 & 0.049 & 0.046 & 0.046 & 0.045 & 0.048 & 0.046 & 0.047 & 0.044 & 0.149 & 0.177 \\ 
  3200 & 7 & 4254 & 0.053 & 0.050 & 0.050 & 0.046 & 0.058 & 0.060 & 0.048 & 0.054 & 0.087 & 0.099 \\ 
  6400 & 8 & 7155 & 0.048 & 0.051 & 0.050 & 0.050 & 0.050 & 0.051 & 0.050 & 0.053 & 0.070 & 0.079 \\  
   \hline
\end{tabular}
\caption{Empirical level of the KS and CvM test statistics for various growth rates $\rho$ of the number of blocks $B = \lfloor 2n^{1-\rho} \rceil$ under the null hypothesis ($\alpha=0.05$).}
\label{tab:H0B}
\end{table}

Table~\ref{tab:H0B} shows that the nominal level of 5\% is correctly achieved when the growth rate of the number of blocks satisfies Assumption~\ref{ass:main}.\ref{ass:growing1}. Specifically, for  $\rho  \in \{6/7, 5/6, 4/5, 3/4\}$, the empirical rejection rates converge toward the nominal level $\alpha=0.05$ as $n$ increases. It is important to note that in this parametric framework, we have $\tau=1/2$, meaning that these values of $\rho$ strictly satisfy the theoretical requirements of our assumption. Conversely, the test fails to maintain the nominal level when Assumption~\ref{ass:main}.\ref{ass:growing1} is violated. When the number of blocks increases too rapidly (\emph{e.g.}, $\rho=1/2$), the number of observations within each block becomes insufficient to ensure the convergence of the local empirical processes, leading to a significant over-rejection of the null hypothesis. 

\subsubsection*{Impact of the number of moment conditions}
Table~\ref{tab:H0pvals} investigates the impact of the growth rate of the number of moment conditions, denoted by $\kappa$, on the test performance. It reports the empirical rejection rates obtained from 5000 samples using the proposed method with $B=\lfloor 2n^{1/4} \rceil$, $p=\lfloor 3 n^{\kappa} \rceil$,   $M=\lfloor 10 n^{3/4} \rceil$  and a nominal level of $\alpha=0.05$.
\begin{table}[ht]
\centering
\begin{tabular}{ccc|cc|cc|cc|cc|cc}
  \hline 
  & &$\kappa$& \multicolumn{2}{c|}{$1/11$}& \multicolumn{2}{c|}{$1/10$}& \multicolumn{2}{c|}{$1/9$}& \multicolumn{2}{c|}{$1/8$}& \multicolumn{2}{c}{$1/7$} \\
$n$ & $B$ & $M$ & KS & CvM & KS & CvM & KS & CvM & KS & CvM & KS & CvM \\ 
  \hline
200 & 6 & 531 & 0.070 & 0.066 & 0.074 & 0.077 & 0.065 & 0.068 & 0.073 & 0.072 & 0.073 & 0.077 \\ 
  400 & 7 & 894 & 0.054 & 0.053 & 0.053 & 0.051 & 0.048 & 0.049 & 0.049 & 0.045 & 0.044 & 0.045 \\ 
  800 & 8 & 1504 & 0.049 & 0.047 & 0.053 & 0.047 & 0.049 & 0.051 & 0.053 & 0.056 & 0.050 & 0.049 \\ 
  1600 & 9 & 2529 & 0.053 & 0.050 & 0.050 & 0.051 & 0.045 & 0.046 & 0.053 & 0.052 & 0.050 & 0.050 \\ 
  3200 & 10 & 4254 & 0.052 & 0.051 & 0.055 & 0.054 & 0.052 & 0.051 & 0.049 & 0.046 & 0.049 & 0.051 \\ 
  6400 & 12 & 7155 & 0.048 & 0.045 & 0.050 & 0.048 & 0.047 & 0.047 & 0.051 & 0.048 & 0.051 & 0.050 \\ 
   \hline
\end{tabular}
\caption{Empirical level of the KS and CvM test statistics for different growth rates $\kappa$ of the number of moment conditions $p = \lfloor 3 n^\kappa \rceil$ under the null hypothesis ($\alpha=0.05$).}
\label{tab:H0pvals}\end{table}
The results show that increasing the number of moment conditions does not alter the empirical level of the test, provided that the growth rate remains within the bounds prescribed by our theoretical assumptions. Indeed, for all tested values of $\kappa \in \{1/11, \dots, 1/7\}$, the empirical rejection rates consistently reach the nominal level of 5\% as the sample size $n$ increases. This highlights the robustness of the procedure regarding the dimension of the moment vector, ensuring that the asymptotic null distribution is well-preserved even when the number of equations grows with the sample size.

\subsubsection*{Impact of the number of Monte Carlo replications}
Table~\ref{tab:H0Monte} illustrates the impact of the number of Monte Carlo replications $M$ on the control of the empirical level. It reports empirical rejection rates for different test statistics and functional bases, based on 5000 samples, using the proposed method with $B=\lfloor 2n^{1/4} \rceil$, $p=\lfloor 3 n^{1/9} \rceil$, $M=\lfloor 10 n^{\xi} \rceil$ and a nominal level of $\alpha=0.05$.

\begin{table}[ht]
\centering
\begin{tabular}{ccc|cc|cc|cc|cc|cc}
  \hline 
  & &$\xi$& \multicolumn{2}{c|}{$1/2$}& \multicolumn{2}{c|}{$2/3$}& \multicolumn{2}{c|}{$3/4$}& \multicolumn{2}{c }{$4/5$}& \multicolumn{2}{c }{$5/6$}\\
  $n$ & $B$ & $p$ & KS & CvM & KS & CvM & KS & CvM & KS & CvM& KS & CvM\\
  \hline
200 & 6 & 5 & 0.095 & 0.104 & 0.075 & 0.081 & 0.067 & 0.069 & 0.078 & 0.078 & 0.070 & 0.071 \\ 
  400 & 7 & 6 & 0.087 & 0.093 & 0.056 & 0.056 & 0.049 & 0.051 & 0.051 & 0.051 & 0.050 & 0.053 \\ 
  800 & 8 & 6 & 0.117 & 0.130 & 0.054 & 0.052 & 0.048 & 0.049 & 0.053 & 0.053 & 0.050 & 0.049 \\ 
  1600 & 9 & 7 & 0.166 & 0.182 & 0.055 & 0.057 & 0.054 & 0.054 & 0.054 & 0.053 & 0.054 & 0.052 \\ 
  3200 & 10 & 7 & 0.254 & 0.276 & 0.060 & 0.059 & 0.053 & 0.052 & 0.047 & 0.052 & 0.049 & 0.049 \\ 
  6400 & 12 & 8 & 0.339 & 0.370 & 0.066 & 0.065 & 0.048 & 0.047 & 0.048 & 0.048 & 0.052 & 0.054 \\ 
   \hline
\end{tabular}\caption{Empirical level of the KS and CvM test statistics for different growth rates $\xi$ of the number of Monte Carlo replications $M = \lfloor 10 n^\xi \rceil$ under the null hypothesis ($\alpha=0.05$).}
\label{tab:H0Monte}
\end{table}
The results show that as long as the Monte Carlo approximation error is negligible, which occurs here for $\xi \in \{2/3, 3/4, 4/5, 5/6\}$, the test correctly achieves the nominal asymptotic level of 5\% as $n$ increases. These empirical findings are in perfect agreement with the theoretical requirements of Theorem~\ref{thm:alternative}. This theorem establishes that the asymptotic null distribution of the test statistic is preserved provided that the numerical approximation error of the moment conditions vanishes at a sufficient rate. Conversely, when the number of replications is insufficient to ensure this order of precision (notably for $\xi = 1/2$), the approximation error remains too large relative to the sampling variance. This leads to a violation of the conditions set forth in Theorem~\ref{thm:alternative}, resulting in a loss of the asymptotic level and a significant over-rejection of the null hypothesis. These findings confirm the theoretical requirement that $M$ must grow fast enough (typically with $\xi > 1/2$) to make the numerical simulation error statistically negligible relative to the sampling distribution's variance, thereby ensuring the validity of the diagnostic procedure.

\subsubsection*{Impact of the dimension of the variables used in clustering}
Finally, we evaluate the impact of the data dimension $d$ on the test performance. Although the data are generated in a higher-dimensional space, the testing procedure is applied to the posterior probabilities, which always lie on a $K$-dimensional simplex. Table~\ref{tab:H0dim} presents the empirical rejection rates for $d \in \{10, 15, 20, 25,  50\}$ based on 5000 samples, using the proposed method with $B=\lfloor 2n^{1/4} \rceil$, $p=\lfloor 3 n^{1/9} \rceil$, $M=\lfloor 10 n^{3/4} \rceil$ and a nominal level of $\alpha=0.05$.
\begin{table}[ht]
\centering
\begin{tabular}{cccc|cc|cc|cc|cc|cc}
  \hline
  &&&d& \multicolumn{2}{c|}{10}& \multicolumn{2}{c|}{15}& \multicolumn{2}{c|}{20}& \multicolumn{2}{c|}{25}& \multicolumn{2}{c}{50} \\
n & p& B& M & KS & CvM & KS & CvM & KS & CvM & KS & CvM & KS & CvM   \\ 
  \hline
200 & 5 & 6 & 531 & 0.108 & 0.109 & 0.161 & 0.159 & 0.243 & 0.242 & 0.309 & 0.308 & 0.673 & 0.665 \\ 
  400 & 6 & 7 & 894 & 0.054 & 0.052 & 0.046 & 0.047 & 0.059 & 0.057 & 0.048 & 0.054 & 0.115 & 0.117 \\ 
  800 & 6 & 8 & 1504 & 0.054 & 0.052 & 0.044 & 0.047 & 0.054 & 0.052 & 0.049 & 0.049 & 0.057 & 0.053 \\ 
  1600 & 7 & 9 & 2529 & 0.050 & 0.047 & 0.048 & 0.047 & 0.054 & 0.053 & 0.050 & 0.050 & 0.049 & 0.049 \\ 
  3200 & 7 & 10 & 4254 & 0.054 & 0.052 & 0.049 & 0.048 & 0.046 & 0.048 & 0.053 & 0.055 & 0.051 & 0.051 \\ 
  6400 & 8 & 12 & 7155 & 0.045 & 0.048 & 0.047 & 0.046 & 0.051 & 0.056 & 0.049 & 0.051 & 0.047 & 0.047 \\ 
   \hline
\end{tabular}
\caption{Empirical level of the KS and CvM test statistics across different data dimensions $d$ using the Bernstein basis under $H_0$ ($\alpha=0.05$).}
\label{tab:H0dim}
\end{table}
Table~\ref{tab:H0dim} demonstrates that the data dimension $d$ does not significantly impact the empirical level of the test. For all considered values of $d$, the rejection rates effectively converge toward the nominal level $\alpha=0.05$ as $n$ increases. This robust behavior is explained by the fact that the complexity of our testing procedure depends on the number of clusters $K$ and the number of moments $p$, rather than the dimension of the variable space $d$. While a larger $d$ may influence the variance of the maximum likelihood estimator by a constant factor, the convergence rate of the estimator remains unchanged. Consequently, the asymptotic properties of the test are preserved regardless of the ambient  dimension of the data, although this dimension can impact the results for small samples.

\subsection{Assessment of the asymptotic level for various clustering models}
In this experiment, we illustrate that the proposed procedure achieves its nominal asymptotic level $\alpha$ across a variety of data distributions. We generate six-dimensional data from mixture models with $K=3$ components and equal proportions. Three types of parametric mixtures are considered to evaluate the versatility of the test: the Gaussian mixture model (GMM), the Poisson mixture model (PMM), and the Bernoulli mixture model (BMM). 

To investigate the impact of component separation, we examine three levels of overlap, corresponding to Bayes classification rates of 0.80, 0.85, and 0.90. These levels are controlled by a scalar $\delta$ applied to the centers:
$\mu_1(\delta)=(2 \delta, \delta, 0,2 \delta, \delta, 0)^\top$, 
$\mu_2(\delta)=(\delta, 0,2 \delta, \delta, 0,2 \delta)^\top$, and 
$\mu_3(\delta)=(0,2 \delta, \delta, 0,2 \delta, \delta)^\top$.

\begin{itemize}
    \item GMM: Defined with centers $\mu_k(\delta)$ and identity covariance matrices. The three classification rates are achieved with $\delta \in \{0.675, 0.780, 0.911\}$.
    \item PMM: Defined with intensity rates $\mu_k(\delta)+\mathbf{1}_6\delta$. The classification rates are achieved with $\delta \in \{0.859, 1.139, 1.552\}$.
    \item BMM: Defined with probability vectors $\text{logit}^{-1}(\mu_k(\delta)-\mathbf{1}_6\delta)$. The classification rates are achieved with $\delta \in \{1.435, 1.692, 2.040\}$.
\end{itemize}

For each sample size, distribution type, and overlap level, we generate $N=5000$ replicates. All tests are conducted using the Bernstein basis with an asymptotic nominal level of $\alpha=0.05$.

\begin{table}[ht]
\centering
\begin{tabular}{cccc|ccc|ccc|ccc}
  \hline
\multicolumn{4}{r}{Error rate}  & \multicolumn{3}{c|}{0.20} &\multicolumn{3}{c|}{0.15} &\multicolumn{3}{c}{0.10} \\
$n$ & $p$ & $B$  & $M$  & GMM & PMM & BMM& GMM & PMM & BMM& GMM & PMM & BMM \\ 
  \hline
200 & 5 & 6 & 531 & 0.348 & 0.134 & 0.209 & 0.213 & 0.151 & 0.264 & 0.354 & 0.265 & 0.486 \\ 
  400 & 6 & 7 & 894 & 0.068 & 0.059 & 0.063 & 0.061 & 0.058 & 0.061 & 0.083 & 0.077 & 0.110 \\ 
  800 & 6 & 8 & 1504 & 0.055 & 0.054 & 0.053 & 0.052 & 0.049 & 0.049 & 0.051 & 0.055 & 0.056 \\ 
  1600 & 7 & 9 & 2529 & 0.051 & 0.050 & 0.050 & 0.055 & 0.051 & 0.053 & 0.054 & 0.048 & 0.047 \\ 
  3200 & 7 & 10 & 4254 & 0.055 & 0.046 & 0.045 & 0.054 & 0.053 & 0.053 & 0.054 & 0.050 & 0.053 \\ 
  6400 & 8 & 12 & 7155 & 0.050 & 0.052 & 0.052 & 0.052 & 0.045 & 0.048 & 0.053 & 0.053 & 0.051 \\ 
   \hline
\end{tabular}
\caption{Empirical level of the test for Gaussian (GMM), Poisson (PMM), and Bernoulli (BMM) mixture models under $H_0$ across different classification error rates ($0.20, 0.15, 0.10$). Rejection rates are calculated over 5000 replicates at a nominal level $\alpha=0.05$ obtained with the KS statistic.}
\label{tab:nominalparam} 
\end{table}

Table~\ref{tab:nominalparam} demonstrates that the proposed testing procedure effectively controls the Type I error across all considered distributions and overlap scenarios. For small sample sizes ($n=200$), like in the previous experiments, we observe an over-rejection of the null hypothesis. However, as $n$ increases, the empirical rejection rates consistently converge toward the nominal level of 5\%, regardless of the underlying nature of the data (continuous, count, or binary) or the degree of cluster overlap. This robust behavior empirically validates our theoretical findings, showing that the asymptotic null distribution is preserved as long as the clustering model is correctly specified.

\subsection{Sensitivity to deviations from the null hypothesis and power results}

We now investigate the power of the procedure to detect model misspecification, focusing on cases where the data distribution deviates from the assumptions required for a well-specified clustering model. Data are generated from two different alternative mixture models, while clustering is performed using a Gaussian mixture model (GMM) with diagonal covariance matrices.

The first case examines a violation of the conditional independence assumption. We generate data from a mixture of three Gaussian components with equal proportions and means $\mu_1(\delta)$, $\mu_2(\delta)$, and $\mu_3(\delta)$ ($\delta=0.675$). While the marginal distributions match the clustering model, the joint distribution does not: the within-component covariance between variables $j$ and $j'$ is defined as $c^{|j-j'|}$. When $c=0$, the model is well-specified, but $c \neq 0$ introduces dependency that the diagonal GMM cannot capture. The second case investigates a situation where the model is misspecified for the posterior classification probabilities, even if it might remain reasonable for hard clustering. Here, data are generated from a mixture of products of Student's $t$-distributions with $df$ degrees of freedom. As $df$ decreases, the tails become heavier, moving further away from the Gaussian assumption used in the clustering process. Table~\ref{tab:H1param} presents the empirical rejection rates for these two scenarios.

\begin{table}[ht]
\centering
\begin{tabular}{cccc|cccc|cccc}
  \hline
  &&&& \multicolumn{4}{c|}{Full Gaussian mixture model} & \multicolumn{4}{c}{Student mixture model}\\
$n$& $p$ & $B$& $M$& $c=0$ & $c=0.25$ & $c=0.50$ & $c=0.75$ & $df=6$& $df=5$& $df=4$& $df=3$ \\ 
  \hline
200 & 5 & 6 & 531 & 0.336 & 0.580 & 0.508 & 0.628 & 0.669 & 0.692 & 0.742 & 0.798 \\ 
  400 & 6 & 7 & 894 & 0.061 & 0.098 & 0.110 & 0.406 & 0.255 & 0.344 & 0.477 & 0.611 \\ 
  800 & 6 & 8 & 1504 & 0.049 & 0.054 & 0.081 & 0.621 & 0.129 & 0.197 & 0.334 & 0.565 \\ 
  1600 & 7 & 9 & 2529 & 0.047 & 0.054 & 0.135 & 0.860 & 0.086 & 0.118 & 0.295 & 0.684 \\ 
  3200 & 7 & 10 & 4254 & 0.044 & 0.064 & 0.362 & 0.973 & 0.088 & 0.155 & 0.378 & 0.913 \\ 
  6400 & 8 & 12 & 7155 & 0.052 & 0.112 & 0.809 & 1.000 & 0.107 & 0.280 & 0.588 & 0.996 \\ 
   \hline
\end{tabular}
\caption{Empirical power of the test under two alternative hypotheses: a GMM with correlated variables ($c$) and a Student's $t$-mixture model ($df$). The case $c=0$ represents the null hypothesis ($H_0$). Rejection rates are computed over 5000 replicates ($\alpha=0.05$).}
\label{tab:H1param}
\end{table}

The results in Table~\ref{tab:H1param} illustrate the power curves of the proposed test. As the data-generating process deviates further from the clustering model assumptions (i.e., as the correlation $c$ increases or the degrees of freedom $df$ decrease), the procedure effectively detects the misspecification with high rejection rates. For instance, with $c=0.75$ or $df=3$, the test reaches a power near 1.000 as the sample size increases.  Conversely, when the alternative model is closer to the null hypothesis (e.g., $c=0.25$ or $df=6$), the detection becomes more challenging. In these cases, a larger sample size $n$ is required to achieve significant power. This highlights the consistency of the test: even subtle deviations from the well-specification assumption can be detected, provided that the number of observations is sufficiently large to distinguish the underlying distribution from the assumed model.

To further evaluate the performance of our approach, we compare it with classical goodness-of-fit (GOF) tests applied directly to the data distribution. We consider the Kolmogorov-Smirnov (KS-MD) and Anderson-Darling (AD-MD) tests, as implemented in the \texttt{gof\_test} function of the \texttt{MDgof} R package \citep{MDgof}. These tests are evaluated using the exact same simulation protocol described previously and for each replicate, the $p$-values of the \texttt{MDgof} tests are computed using $B=100$ Monte Carlo simulations, as required by the package to handle the estimation of unknown parameters. The empirical rejection rates for the correlated Gaussian and Student's $t$ scenarios are reported in Table~\ref{tab:comp_correlation} and Table~\ref{tab:comp_student}, respectively.

\begin{table}[ht]
\centering
\begin{tabular}{c|cc|cc|cc|cc}
  \hline
$n$ & \multicolumn{2}{c|}{$df=6$}  & \multicolumn{2}{c|}{$df=5$}  & \multicolumn{2}{c|}{$df=4$}  & \multicolumn{2}{c}{$df=3$} \\ 
& KS-MD & AD-MD& KS-MD & AD-MD&KS-MD & AD-MD&KS-MD & AD-MD\\
  \hline
200 & 0.003 & 0.024 & 0.028 & 0.095 & 0.020 & 0.224 & 0.036 & 0.584 \\ 
  400 & 0.009 & 0.046 & 0.051 & 0.169 & 0.044 & 0.311 & 0.153 & 0.761 \\ 
  800 & 0.006 & 0.046 & 0.131 & 0.211 & 0.227 & 0.400 & 0.771 & 0.868 \\ 
  1600 & 0.004 & 0.043 & 0.486 & 0.273 & 0.794 & 0.500 & 1.000 & 0.941 \\ 
  3200 & 0.000 & 0.042 & 0.965 & 0.351 & 0.999 & 0.618 & 1.000 & 0.977 \\ 
  6400 & 0.004 & 0.041 & 1.000 & 0.492 & 1.000 & 0.724 & 1.000 & 0.992 \\ 
   \hline
\end{tabular}
\caption{Empirical rejection rates for KS-MD and AD-MD tests (from \texttt{MDgof}) under the correlated Gaussian mixture model. The column $c=0$ represents the null hypothesis. Rejection rates are based on 5000 replicates with $\alpha=0.05$ and $B=100$ internal simulations for $p$-value estimation.}
\label{tab:comp_correlation}
\end{table}

\begin{table}[ht]
\centering
\begin{tabular}{c|cc|cc|cc|cc}
  \hline
$n$ & \multicolumn{2}{c|}{$df=6$}  & \multicolumn{2}{c|}{$df=5$}  & \multicolumn{2}{c|}{$df=4$}  & \multicolumn{2}{c}{$df=3$} \\ 
& KS-MD & AD-MD& KS-MD & AD-MD&KS-MD & AD-MD&KS-MD & AD-MD\\
  \hline
 200 & 0.010 & 0.112 & 0.006 & 0.115 & 0.014 & 0.113 & 0.023 & 0.139 \\ 
  400 & 0.023 & 0.222 & 0.012 & 0.221 & 0.025 & 0.186 & 0.045 & 0.260 \\ 
  800 & 0.059 & 0.434 & 0.055 & 0.375 & 0.100 & 0.360 & 0.179 & 0.517 \\ 
  1600 & 0.208 & 0.620 & 0.223 & 0.585 & 0.384 & 0.580 & 0.615 & 0.822 \\ 
  3200 & 0.776 & 0.803 & 0.770 & 0.756 & 0.931 & 0.807 & 0.994 & 0.956 \\ 
  6400 & 0.999 & 0.892 & 1.000 & 0.922 & 1.000 & 0.949 & 1.000 & 0.996 \\  
  \hline
\end{tabular}
\caption{Empirical rejection rates for KS-MD and AD-MD tests (from \texttt{MDgof}) under the Student's $t$ mixture model ($df$ degrees of freedom). Results are based on 5000 replicates with $\alpha=0.05$ and $B=100$ internal simulations.}
\label{tab:comp_student}
\end{table}

A critical observation from Table~\ref{tab:comp_correlation} is that under the null hypothesis ($c=0$), these tests fail to reach the nominal level of 5\%. For instance, the KS-MD and CvM-MD tests consistently show rejection rates near 0.005, indicating a highly conservative behavior. While the AD-MD test is closer to the nominal level (around 0.04), none of the methods perfectly control the Type I error in this finite-sample setting. Regarding the alternative hypotheses, although the power increases with the sample size $n$ and the degree of misspecification ($c$ or $df$), these results are less conclusive than those of our proposed method. The lack of proper level control under $H_0$ makes the interpretation of the power results for $c > 0$ or $df < \infty$ more difficult, as the tests might either over-penalize or under-detect deviations depending on the chosen statistic. This reinforces the advantage of our approach, which provides a more reliable internal diagnostic specifically tailored for the posterior probabilities in model-based clustering.

To evaluate if the stability observed under the null hypothesis carries over to the detection of misspecification for our proposed approach, we repeated the reproducibility experiment in the most challenging Student's $t$-mixture scenario ($df=3$). For each sample size $n$, we computed the majority decision (rejection of $H_0$) over 50 random splits across 5000 independent datasets. Table~\ref{tab:RepeatH1} presents the mean consensus rate and its standard deviation.

\begin{table}[ht]
\centering
\small
\begin{tabular}{cccc|cc|cc|cc|cc|cc}
  \hline
& & & & \multicolumn{2}{c|}{Max} & \multicolumn{2}{c|}{Sum} & \multicolumn{2}{c|}{KS} & \multicolumn{2}{c|}{CvM} & \multicolumn{2}{c}{AD} \\
$n$ & $B$ & $p$ & $M$ & mean & sd & mean & sd & mean & sd & mean & sd & mean & sd \\  
  \hline
200 & 6 & 5 & 531 & 0.980 & 0.060 & 0.979 & 0.065 & 0.858 & 0.155 & 0.845 & 0.157 & 0.971 & 0.079 \\ 
400 & 7 & 6 & 894 & 0.930 & 0.112 & 0.930 & 0.113 & 0.774 & 0.156 & 0.770 & 0.157 & 0.909 & 0.130 \\ 
800 & 8 & 6 & 1504 & 0.851 & 0.146 & 0.881 & 0.138 & 0.741 & 0.147 & 0.748 & 0.152 & 0.838 & 0.157 \\ 
1600 & 9 & 7 & 2529 & 0.758 & 0.158 & 0.888 & 0.137 & 0.771 & 0.159 & 0.793 & 0.162 & 0.838 & 0.158 \\ 
3200 & 10 & 7 & 4254 & 0.812 & 0.157 & 0.975 & 0.061 & 0.913 & 0.125 & 0.934 & 0.109 & 0.953 & 0.092 \\ 
6400 & 12 & 8 & 7155 & 0.957 & 0.075 & 1.000 & 0.004 & 0.997 & 0.015 & 0.998 & 0.011 & 0.999 & 0.008 \\ 
   \hline
\end{tabular}
\caption{Stability of the majority decision (rejection rate) under $H_1$ (Student's $t$-mixture, $df=3$, Bernstein basis). Mean and sd reflect the consistency of the test's power across 50 random splits per sample.}
\label{tab:RepeatH1}
\end{table}

The results confirm that the testing procedure maintains a robust consensus even under the alternative hypothesis. While the variability (sd) slightly increases for intermediate sample sizes where the power is in transition, the consensus reaches perfect stability (sd near 0) as $n$ increases, with the \textit{sum} and \textit{AD} statistics showing the fastest convergence toward a certain rejection of the false model.

\subsection{Sensitivity to deviations from the null hypothesis and power results for nonparametric mixture models}

In this section, we extend our numerical investigation to the nonparametric framework, specifically focusing on the conditional independence assumption within a multivariate mixture model. In such a case, $\btheta$ consists of a parametric part (proportions) and a nonparametric part (univariate densities for each component and variable). This scenario is particularly challenging as it lies outside the strict theoretical scope previously established. Indeed, while estimators for univariate densities in nonparametric product mixture models are available \citep{LevineBiometrika2011}, establishing their convergence rates and consequently the influence of the estimation error on the test's asymptotic behavior is still an open problem, leading that the converge rate of $\sup_{1\leq i\leq n} \|\Psi_{\model,p}(\bX_i;\stheta) -  \Psi_{\model,p}(\bX_i;\wtheta)\|_2$ to zero is unknown.

The data-generating process follows the setup described in the previous section: we consider a mixture of three Gaussian components with means $\mu_1(\delta)$, $\mu_2(\delta)$, and $\mu_3(\delta)$ ($\delta=0.675$). To evaluate the power of the test against violations of the conditional independence, we introduce a within-component correlation $c$ between variables $j$ and $j'$, defined as $c^{|j-j'|}$. We apply our testing procedure using the same tuning parameters as in the parametric study ($B=\lfloor 2n^{1/4} \rceil$, $p=\lfloor 3 n^{1/9} \rceil$, and $M=\lfloor 10 n^{3/4} \rceil$).

\begin{table}[ht]
\centering
\begin{tabular}{cccc|cccc}
  \hline
$n$ & $p$ & $B$ & $M$ & $c=0$ & $c=0.25$ & $c=0.50$ & $c=0.75$ \\ 
  \hline
200 & 5 & 6 & 531 & 0.274 & 0.354 & 0.384 & 0.522 \\ 
400 & 6 & 7 & 894 & 0.060 & 0.082 & 0.085 & 0.330 \\ 
800 & 6 & 8 & 1504 & 0.053 & 0.056 & 0.062 & 0.585 \\ 
1600 & 7 & 9 & 2529 & 0.050 & 0.054 & 0.098 & 0.877 \\ 
3200 & 7 & 10 & 4254 & 0.049 & 0.064 & 0.276 & 0.958 \\ 
6400 & 8 & 12 & 7155 & 0.052 & 0.079 & 0.702 & 0.999 \\ 
   \hline
\end{tabular}
\caption{Empirical level and power of the test for the nonparametric mixture model under deviations from conditional independence (correlation $c$). The case $c=0$ corresponds to the null hypothesis $H_0$. Results are based on 5000 replicates ($\alpha=0.05$).}
\label{tab:H1nonparam}
\end{table}

The results in Table~\ref{tab:H1nonparam} demonstrate that the procedure remains effective even in this nonparametric setting. When the model is well-specified ($c=0$), the empirical rejection rate correctly approaches the nominal level $\alpha=0.05$ as $n$ increases, suggesting that the lack of explicit theoretical convergence rates for the univariate density estimators does not hinder the test's level control in practice. Furthermore, the test exhibits significant power in detecting deviations from the null hypothesis. As the correlation $c$ increases, the rejection rates rise consistently with the sample size, reaching near-certainty ($0.999$) for $c=0.75$ at $n=6400$. This highlights the versatility of the proposed diagnostic tool, which appears robust enough to handle nonparametric specifications where the independence assumption is violated.

\section{Applications on real data}
 \subsection{Congressional Voting Records}
We consider the Congressional Voting Records data set \citep{schlimmer1987concept}, which contains the votes of each of the $n=435$ members of the U.S. House of Representatives on 16 key issues. For each vote, three outcomes are recorded: yea, nay, or unknown disposition. The data are modeled using a mixture of products of multinomial distributions \citep{Goo74}.
Parameter estimation is carried out via maximum likelihood, and model selection is based on the BIC criterion \citep{Schwarz:78}, which selects $K=4$ components. Parameter estimation is performed with the function \texttt{VarSelCluster}  the {\sc R} package VarSelLCM \citep{bioinfo}. 

The proposed procedure, which enables a goodness-of-fit test for the distribution of posterior classification probabilities, is implemented using the tuning parameters described in Section~\ref{sec:num}. Specifically, the procedure is run with  a nominal level $\alpha=0.05$ and the tuning parameters described in Section~\ref{sec:prac:imp}. The observed test statistic for the maximum statistic is $y^\dagger_{\model,n,p,\widehat{\btheta}n}=1800.769$, while the corresponding quantile is $q_{\mathcal{X}^2_p,1-\alpha_n}=17.60$. Table~\ref{tab:appli1pvals} gives the pvalues obtained by the other four test statistics.
\begin{table}[h]
    \centering
    \begin{tabular}{cccc}
    \hline
        sum & Kolmogorov-Smirnov & Cramér-von Mises & Anderson-Darling \\
    \hline
        $<10^{-4}$ & 0.0064 & 0.0010 & $<10^{-4}$ \\ 
    \hline
    \end{tabular}
    \caption{P-values obtained for goodness-of-fit testing of the distribution of the posterior probabilities of classification obtained by a mixture of product of multinomial distributions on the Congressional Voting Records data.}
    \label{tab:appli1pvals}
\end{table}

As a result, the procedure rejects the hypothesis that the posterior classification probabilities arise from a mixture of products of multinomial distributions. Figure~\ref{fig:Congressional} shows the QQplots  comparing the empirical distributions of the posterior classification probabilities for each component with their theoretical distributions under the fitted mixture model. This figure shows that the both distributions are not similar which is in agreement with the conclusion of the testing procedure. Hence, the mixture model of product of multinomial distributions seems not adequate for clustering this data set.

\begin{figure}[htp]
    \centering
    \includegraphics[width=0.95\linewidth]{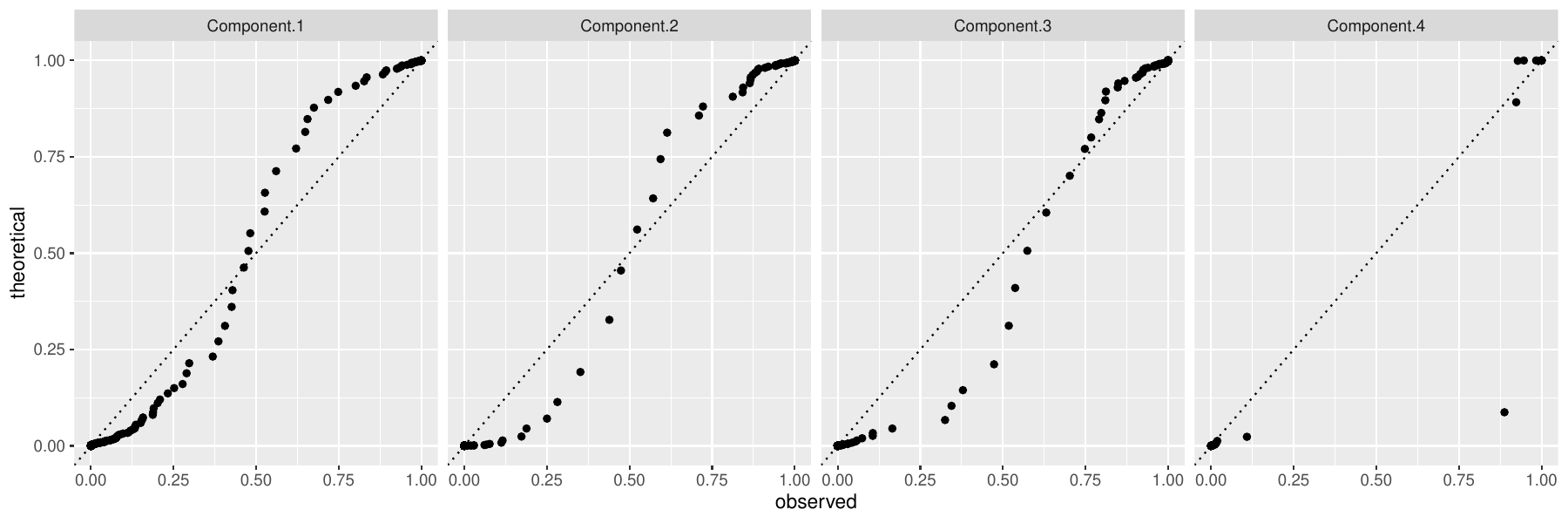}
    \caption{Quantile-quantile plot comparing the empirical distribution of the posterior classification probabilities for each component with their theoretical distributions under the fitted mixture model.}
    \label{fig:Congressional}
\end{figure}

Since the data are categorical, the marginal distribution of each variable is correctly specified under this model class. However, rejection of the null hypothesis appears reasonable, as several key issues relate to the same underlying topic. For example, three votes pertain to children: \emph{handicapped-infants}, \emph{religious-groups-in-schools}, and \emph{education-spending}. These three variables exhibit strong dependence: the chi-square test rejects independence between each pair, with p-values less than $2.2 \times 10^{-16}$. Hence, this dependency between variables is not only explained by the latent variable of the component memberships since our procedure rejects the null hypothesis. We therefore conclude that the posterior classification probabilities are not well approximated, due to the model's assumption of conditional independence across variables within components. It is important to note that directly testing this assumption from the data is challenging without relying on the proposed procedure. Indeed, such a test would require knowledge of the true component memberships, which are not observable. Instead, only estimators of these memberships, derived from the model under evaluation, are available.

\subsection{Graft-versus-Host Disease}
We consider the  Graft-versus-Host Disease data \citep{brinkman2007high} that gathers two samples of this flow cytometry data, one from a patient with the Graft-versus-Host Disease (9083 observations), and the other from a control patient  (6809 observations). The Graft-versus-host disease  is a severe complication that can occur following hematopoietic stem cell transplantation.   Each observation includes four continuous variables that correspond to biomarkers. Hence, the data set is composed of $n=15892$ observations described by 4 continuous variables while the information related to the patient is not used during clustering. 

First, clustering is achieved by a bi-component mixture model of Gaussian distributions with diagonal covariance matrices. The proposed procedure   is run with  a nominal level $\alpha=0.05$ the tuning parameters described in Section~\ref{sec:prac:imp}. The observed test statistic for the maximum statistic is $y^\dagger_{\model,n,p,\widehat{\btheta}n}=22.13$ and the corresponding quantile is $q_{\mathcal{X}^2_p,1-\alpha_n}=13.05$. Table~\ref{tab:appli1pvals} gives the pvalues obtained by the other four test statistics.
\begin{table}[h]
    \centering
    \begin{tabular}{cccc}
    \hline
        sum & Kolmogorov-Smirnov & Cramér-von Mises & Anderson-Darling \\
    \hline
        $<10^{-4}$ & 0.0008 &  $<10^{-4}$ & $<10^{-4}$ \\ 
    \hline
    \end{tabular}
    \caption{P-values obtained for goodness-of-fit testing of the distribution of the posterior probabilities of classification obtained by bi-component Gaussian mixture in the Graft-versus-Host Disease data.}
    \label{tab:appli1pvals}
\end{table}

Therefore, at the asymptotic level 0.05, we reject the hypothesis claiming that the posterior probabilities of classification arises from a Gaussian mixture model with diagonal covariance matrices. 

\section{Conclusion} \label{sec:concl} 
In this paper, we introduced a procedure that evaluates the relevance of a mixture model used for clustering.
This procedure consists of a goodness-of-fit test of the distribution of the posterior probabilities of classification, thereby directly considering the clustering aim.
By focusing directly on the posterior probabilities of classification, all the nature of data can be analyzed by the same procedure. In addition, the mixture model can be considered in a parametric or non-parametric framework.
The procedure does not necessitate any additional parameter estimation since it relies on the posterior probabilities of classification computed with an estimator of the model parameters. 

The proposed procedure is based on $p$ functional moments. If $p$ is fixed, then only mild assumptions are required on the functions, but there is no guarantee to detect all the alternatives. Therefore, we propose to allow $p$ to grow with the sample size at an appropriate rate. In this context, more restrictive conditions should be satisfied by the basis functions.

Other goodness-of-fit testing procedures could be considered to investigate the relevance of modeling the distribution of the posterior probability of classifications. For instance, some extensions of the Kolmogorov-Smirnov test with estimated parameters \citep{braun1980simple} could be considered. However, note that the distribution of the posterior probabilities of classification is generally not explicit, but it can be approximated using numerical methods (\emph{e.g.}, Monte Carlo methods). Such a procedure is promising if $K=2$ because the Kolmogorov-Smirnov test was developed for univariate data, and because the posterior probabilities of classification are defined on the simplex, meaning that their dimension is one when $K=2$. If $K$ is greater than two, extensions of the Kolmogorov-Smirnov test to multivariate data could be considered as an alternative to the proposed procedure. Such an approach would have the advantage of avoiding the choice of the dimension $p$ as well as the choice of basis functions. However, such extensions remain challenging and could be developed in future work.

A natural question arises when the mixture model $m$ is selected using a criterion such as the Bayesian Information Criterion (BIC) before applying the proposed test. In this work, the asymptotic theory is developed under a fixed model $m$. However, provided that the number of candidate models is finite, the consistency of the BIC ensures that the selected model remains the same with probability tending to one as $n \to \infty$. Under the null hypothesis, the BIC will select the true model, thus aligning the post-selection test with our fixed-model asymptotic framework. Under alternatives, the BIC will consistently select the "best" model among the candidates (see \citet{KeribinSan2000}), and our procedure then serves to assess whether this chosen model is nonetheless misspecified for clustering. The case where the number of candidate models grows with $n$ is a complex problem of post-selection inference that falls beyond the scope of this paper but remains an interesting avenue for future research.

The approach is developed by assuming independence between the observations. Extension to dependent data could be considered, in order to consider, for instance, hidden Markov chains with a finite number of states. In such a case, the empirical likelihood statistics should not be computed directly. Indeed, for dependent data, the empirical likelihood ratio does not converge to a chi-square random variable with $p$ degrees of freedom but to a weighted sum of $p$ independent chi-square random variables with 1 degree of freedom. To encompass this problem, blocking techniques could be considered as proposed by \citep{kitamura1997empirical} in the case of weakly dependent data and combined with our approach where we also manage blocks strategies for the nuisance parameter. The proof of Theorem \ref{thm:niveau} required a Berry-Essen concentration bound which is obtained under independence assumption in our context. Concentration results for weakly dependent data and growing dimension can be applied to control this step, but certain technical arguments would need to be adapted using results in \cite{MR3502600} and \cite{chang2024central}. 

\section*{Funding}
The authors have no funding to report.

\section*{Data availability}
The data underlying the illustration of this article were derived from sources in the public domain: the Congressional Voting Records data set \citep{schlimmer1987concept} is available on the UC Irvine Machine Learning Repository\footnote{https://archive.ics.uci.edu/dataset/105/congressional+voting+records} and  and 
the Graft-versus-Host Disease data \citep{brinkman2007high} is availble in the \textsc{R} package mclust \citep{scrucca2016mclust}.


\bibliographystyle{abbrvnat}
 \bibliography{biblio}
\begin{appendices} 
\section{Proof of the main results} \label{sec:appproof}

To state these results, we introduce four technical lemmas, presented here and proved in Section \ref{B.2}. For notational convenience, we recall that for $\boldsymbol{r}={r_1,\ldots,r_B}$, $r^{\dagger}$ denotes the maximum of $r_b$ over $b=1,\dots,B$ of the scalar magnitude of $X_b$, that is absolute value for real quantities, Euclidean norm for vectors, spectral norm for matrices.

Lemma~\ref{lem:controlMax} gives the stochastic order of $$V_{\est}=\max_{1\leq i \leq n_b} \|\Psi_{\model,p}(\bX_i^{(b)};\wtheta) \|_2. $$ 
\begin{lemma} \label{lem:controlMax}
Under the assumptions of Theorem \ref{thm:niveau},  we have 
 $V_{\est}=  o_\mathbb{P}(n^{\rho/2}p^{-1})$.
\end{lemma}
Lemma~\ref{lem:covmatrix} gives a control of the stochastic order of the matrix   defined as the  difference between the empirical covariance matrix $ \bS_{\est}$  and the theoretical covariance matrix $\bSigma_{\model,p}$ as well as  a control of the stochastic order of the matrix $\bGamma_{\est}$ defined as the difference between the inverse of these matrices where
$$\bS_{\any}=\frac{1}{n_b}\sum_{i=1}^{n_b} \Psi_{\model,p}(\bX_i^{(b)};\btheta)\Psi_{\model,p}(\bX_i^{(b)};\btheta)^\top$$
and
$$\bGamma_{\any}=\bS_{\any}^{-1} -  \bSigma_{\model,p}^{-1}.$$

\begin{lemma} \label{lem:covmatrix}
Under the assumptions of Theorem \ref{thm:niveau}, there exists $\vartheta_1:=\rho/2 -\kappa(1+2(1+r_0)/q_0)>0$ such that
\begin{equation*}
 \left\| \bS_{\est} -  \bSigma_{\model,p}\right\|_{sp} =  O_{\mathbb{P}}(n^{-\vartheta_1}) \text{ and }   \|\bGamma_{\est} \|_{sp} =  O_{\mathbb{P}}(n^{-\vartheta_1}).
\end{equation*}
 \end{lemma}

Lemma~\ref{lem:lagrange}  develops stochastic order of the Lagrange multipliers $\blambda_{\est} $ introduced in the Empirical Likelihood.  
\begin{lemma}\label{lem:lagrange} Considering $\bZ_{\est}$ defined in \eqref{Sec3:Zany}, under the assumptions of Theorem \ref{thm:niveau}, we have
$$
\| \bZ_{\est}\|_2=O_\mathbb{P}(p^{1/2})
$$
and
$$\blambda_{\est}  = n_b^{-1/2}\left(\bS_{\est}\right)^{-1} \bZ_{\est}  + \bbeta_{\est},$$ 
with  $ \|\blambda_{\est}\|_2  =O_\mathbb{P}( n^{-\rho/2} p^{1/2})$ and $ \left\|\bbeta_{\est}\right\|_2 =  O_\mathbb{P}(n^{-\vartheta_2})$ where $\vartheta_2:=\rho -\kappa(5/2+3r_0/q_0)>0$.
\end{lemma}
As a direct consequence of Lemma~\ref{lem:lagrange}, we have  $\left\|\bbeta_{\est}\right\|_2 =  o_\mathbb{P}(\|\blambda_{\est}\|_2 )$.\\

Lemma~\ref{lem:vectorrate} permits to state a stochastic order of the two quantities $Z^\dagger_{n,p,\wtheta}$ and $\Gamma^\dagger_{n,p,\wtheta}$.

\begin{lemma}\label{lem:vectorrate}
Under the assumptions of Theorem \ref{thm:niveau}, we have
$Z^\dagger_{n,p,\wtheta}=O_\mathbb{P}(n^{\kappa/2} + \ln^{1/2} n)$   and $\Gamma^\dagger_{n,p,\wtheta}=O_\mathbb{P}(n^{-\vartheta_3})$ where $\vartheta_3:=\vartheta_1 - (1-\rho)/q_0>0$.
\end{lemma}

\begin{proof}[Proof of Theorem~\ref{thm:niveau}]
We need to show that 
\begin{equation}\label{Appendix A: eq1}
\max_{1\leq b\leq B} |Y_{\est}-W_{\true}|=o_\mathbb{P}(1),    
\end{equation}
where $Y_{\est}$ and $W_{\true}$ are defined respectively in \eqref{Sec2: Yany} and \eqref{Sec3: W}. From the previously stated lemmas, we can show \eqref{Appendix A: eq1}. Note that the first part of the proof  generalizes some results stated by \citet{owen2001empirical} to the case of growing dimension and provides more accurate stochastic orders of remainders terms in order to be able to work with the maximum of the statistics over the $B$ blocks.  Let $U_{\any,i}=\Psi_{\model,p}(\bX_i^{(b)};\btheta)^\top\blamany$. For any $\btheta$ such that for any $\max_{1\leq i \leq n_b}|U_{\any,i}|=o_\mathbb{P}(1)$,  by a third order Taylor expansion of the $\ln(1+u)$ around $u=0$, we have  
\begin{equation}\label{eq:taylorxi0theta}
Y_{\any}  = 2n_b^{1/2} \blamany ^\top \bZany  
- n_b  \blamany^\top \bS_{\any} \blamany + \eta_{\any} ,
\end{equation}
with
\begin{equation}\label{A2:eta}
\eta_{\any} \leq O(\|\blamany\|_2^3) \sum_{i=1}^{n_b} \left\| \Psi_{p,\model}(\bX_i^{(b)};\btheta ) \right\|_2^3.
\end{equation}
Noting that for any $\btheta$,  
\begin{equation}\label{eq:Uborne}
\max_{1\leq i \leq n_b} |U_{\any,i}| \leq \|\blamany\|_2 V_{\any},
\end{equation}
combining Lemmas~\ref{lem:controlMax} and \ref{lem:lagrange}  implies  $$\max_{1\leq i \leq n_b} |U_{\est,i}|  =o_\mathbb{P}(p^{-1/2}),$$ leading that Taylor expansion defined by \eqref{eq:taylorxi0theta} can be considered at $\btheta=\wtheta$. We now establish the stochastic order of $\eta_{\est}$. Assumption~\ref{ass:main}-\ref{ass:growing2} implies \eqref{eq:moment}, so by Hölder's inequality    for any integer $s$ such that $s\leq q_0$, we have
\begin{equation}\label{eq:momentS}
\mathbb{E} \|\Psi_{\model,p}(\bX_i^{(b)};\stheta)\|_2^{s}=O(p^{s/2+sr_0/q_0}).
\end{equation}
In addition, Minkowski's inequality implies that
\begin{equation} \label{eq:Minkowski}
 \|\Psi_{\model,p}(\bX_i^{(b)};\wtheta)\|_2^{s} \leq 2^{s/2} \left(\|\Psi_{\model,p}(\bX_i^{(b)};\stheta)\|_2^s + \|\Psi_{\model,p}(\bX_i^{(b)};\wtheta) - \Psi_{\model,p}(\bX_i^{(b)};\stheta) \|_2^s \right).
\end{equation}
Since $ \max_{1\leq i \leq n} \|\Psi_{\model,p}(\bX_i;\stheta) - \Psi_{\model,p}(\bX_i;\wtheta)\|_2=O_\mathbb{P}(n^{-\tau}p^{1/2})$ and $n^{-\tau}p^{1/2}=o(p^{3/2+3r_0/q_0})$,  we have
$$\mathbb{E} \|\Psi_{\model,p}(\bX_1;\wtheta)\|_2^{3}=O(p^{3/2+3r_0/q_0}).$$
Law of Large Number and Assumption~\ref{ass:main}-\ref{ass:growing2} imply that 
$$\frac{1}{n_b}\sum_{i=1}^{n_b}  \|\Psi_{\model,p}(\bX_i^{(b)};\wtheta)\|_2^3=O_\mathbb{P}(n^{\kappa(3/2+3r_0/q_0)}).$$ 
Therefore, the control of norm of Lagrange multipliers stated by Lemma~\ref{lem:lagrange} leads to
$$
\eta_{\est} = O_\mathbb{P}(n^{-\rho/2+\kappa(3/2+3r_0/q_0)}).
$$
Note that Assumption~\ref{ass:main}-\ref{ass:growing3} implies that $\lim_{n\to\infty} n^{-\rho/2+\kappa(3+3r_0/q_0)}=0$ and thus $\eta_{\est}=o_\mathbb{P}(1)$. Now, using Lemma~\ref{lem:lagrange}, to replace  the Lagrange multipliers by their asymptotic developments in the right-hand side of \eqref{eq:taylorxi0theta} evaluated at $\wtheta$ leads to
\begin{equation}\label{A2:Ytrue}
Y_{\est}  = W_{\est} +  \varepsilon_{1,\est}  + \varepsilon_{2,\est}   + \eta_{\est},
\end{equation}
with 
\begin{equation}\label{A2: eps}
\left\{
\begin{array}{rl}
\varepsilon_{1,\est} &= \bZ_{\est}^{\top}\bGamma_{\est} \bZ_{\est}  \\
 \varepsilon_{2,\est} &=  - n_b \bbeta_{\est}^{\top}  \bS_{\est} \bbeta_{\est}\\
\end{array}
\right.   .
\end{equation}

The triangular inequality implies
$$
\max_{1\leq b \leq B}|Y_{\est}  -   W_{\est}| \leq  \varepsilon^{\dagger}_{1,\estall}  + \varepsilon^{\dagger}_{2,\estall}   + \eta^{\dagger}_{\estall}.
$$
and,
$$
|\varepsilon_{1,\est}|\leq \|\bGamma_{\est}  \|_{sp} \|\bZ_{\est} \|_2^2,
$$
leading that
 $$\varepsilon^{\dagger}_{1,\estall}\leq Z_{n,p,\wtheta}^{\dagger 2}\Gamma_{n,p,\wtheta}^\dagger,$$
 where $Z_{n,p,\wtheta} =\max_{b=1,\ldots,B}\| \bZ_{\est} \|_2$
 Using Lemma~\ref{lem:vectorrate} and Assumption~\ref{ass:main}-\ref{ass:growing3}, we have  $$\varepsilon^{\dagger}_{1,\estall} = O_\mathbb{P}(n^{ -\vartheta_3}(n^{\kappa} + \ln n)).$$ 
Note that $n^{-\vartheta_3}\ln n = o(n^{-\vartheta_1-\kappa+1/q_0})$ and $n^{-\vartheta_3 + \kappa} = o(1)$, thus, using Assumption~\ref{ass:main}-\ref{ass:growing3}, we have
$$ \varepsilon^{\dagger}_{1,\estall} = o_\mathbb{P}(1).$$
Triangular inequality implies that
$$
|\varepsilon_{2,\est}|\leq n_b \|\bS_{\est}\|_{sp} \|\bbeta_{\est}\|_2^2.
$$
Using Lemma~\ref{lem:covmatrix} and Assumption~\ref{ass:main}-\ref{ass:cov} imply that $\|\bS_{\est}\|_{sp} =O_\mathbb{P}(1)$. Combining this with  Lemma~\ref{lem:lagrange} ensures that
$$\varepsilon_{2,\est}  = O_\mathbb{P}(n^{\rho-2\vartheta_2}).$$ 
Thus, using  the union bound, we have
 $$
\varepsilon^{\dagger}_{2,\estall}=O_\mathbb{P}(n^{1-2\rho-2\vartheta_2}).
$$
We have $1-2\rho-2\vartheta_2=1-4\rho+\kappa(5+6r_0/q_0)$, therefore $n^{1-2\rho-2\vartheta_2}=n^{1-3\rho}n^{-\rho +\kappa(5+6r_0/q_0)}$. Note that by Assumption~\ref{ass:main}-\ref{ass:growing2}, $\rho>1/3$ leading that $n^{1-3\rho}=o(1)$ and $n^{-\rho +\kappa(5+6r_0/q_0)}=o(1)$ by Assumption~\ref{ass:main}-\ref{ass:growing3}. Thus,
 $$
\varepsilon^{\dagger}_{2,\estall}=o_\mathbb{P}(1).
$$
Using \eqref{A2:eta}, we have
$$
 \eta^{\dagger}_{\estall} \leq \max_{1\leq b \leq B} \sum_{i=1}^{n_b} \|\Psi_{p,\model}(\bX_i^{(b)};\wtheta ) \|_2^3 \max_{1\leq b \leq B}\|\lambda_{\est}\|_2^3 .
$$
Using the union bound, we have
$$
\mathbb{P}(\max_{1\leq b \leq B} \sum_{i=1}^{n_b} \|\Psi_{p,\model}(\bX_i^{(b)};\stheta ) \|_2^3 \geq \varepsilon) \leq \sum_{b=1}^{B} \mathbb{P}(\sum_{i=1}^{n_b} \|\Psi_{p,\model}(\bX_i^{(b)};\stheta ) \|_2^3\geq \varepsilon).
$$
Markov's inequality implies that for any $s>0$  
$$
\mathbb{P}(\sum_{i=1}^{n_b} \|\Psi_{p,\model}(\bX_i^{(b)};\stheta ) \|_2^3\geq \varepsilon) \leq n_b \frac{\mathbb{E} \|\Psi_{p,\model}(\bX_i^{(b)};\stheta) \|_2^{3s} }{\varepsilon^s}.
$$
Since $\sum_{b=1}^{B}n_b = n  $, using the previous inequality with $s=q_0/3$, we have
$$
\mathbb{P}(\max_{1\leq b \leq B} \sum_{i=1}^{n_b} \|\Psi_{p,\model}(\bX_i^{(b)};\stheta) \|_2^3 \geq \varepsilon) \leq \frac{n \mathbb{E} \|\Psi_{p,\model}(\bX_i^{(b)};\stheta ) \|_2^{q_0} }{\varepsilon^{q_0/3}}.
$$
Using the order of $\mathbb{E} \|\Psi_{p,\model}(\bX_i^{(b)};\stheta ) \|_2^{q_0} $ given by \eqref{eq:moment}, we have that there exists $\tilde{C}>0$ such that
$$
\mathbb{P}(\max_{1\leq b \leq B} \sum_{i=1}^{n_b} \|\Psi_{p,\model}(\bX_i^{(b)};\stheta ) \|_2^3 \geq \varepsilon) \leq  \tilde{C} \frac{ n^{q_0(3/q_0 + \kappa(3/2 + 3 r_0/q_0 ))/3  }}{\varepsilon^{q_0/3}}.
$$
Therefore, we have
$$
\max_{1\leq b \leq B} \sum_{i=1}^{n_b} \|\Psi_{p,\model}(\bX_i^{(b)};\stheta) \|_2^3 = O_\mathbb{P}(n^{3/q_0 + \kappa(3 + 3 r_0/q_0)}).
$$
Similarly, we can show that
$$
\max_{1\leq b \leq B} \sum_{i=1}^{n_b} \|\Psi_{p,\model}(\bX_i^{(b)};\stheta ) \|_2^2 = O_\mathbb{P}(n^{2/q_0 + \kappa(2 + 2 r_0/q_0)}).
$$
Since for any $1\leq b \leq B$ and $1\leq i \leq n_b$, we have
$$
|\|\Psi_{\model,p}(\bX_i^{(b)};\wtheta)\|_2-\|\Psi_{\model,p}(\bX_i^{(b)};\stheta))\|_2|\leq \max_{1\leq i\leq n} \|\Psi_{\model,p}(\bX_i;\wtheta) - \Psi_{\model,p}(\bX_i;\stheta) \|_2,
$$
then there exists a positive constant $C$ such that

\begin{equation*}
\begin{aligned}
\bigl|
\|\Psi_{\model,p}(\bX_i^{(b)};\wtheta)\|_2^3
-
\|\Psi_{\model,p}(\bX_i^{(b)};\stheta)\|_2^3
\bigr|
\\
\leq\;
C \,
\|\Psi_{\model,p}(\bX_i^{(b)};\stheta)\|_2^2
\max_{1\leq i\leq n}
\bigl\|
\Psi_{\model,p}(\bX_i;\wtheta)
-
\Psi_{\model,p}(\bX_i;\stheta)
\bigr\|_2 .
\end{aligned}
\end{equation*}

Hence, by Assumption~\ref{ass:main}-\ref{ass:cvtheta2}, 
\begin{multline*}
|\max_{1\leq b \leq B} \sum_{i=1}^{n_b} \|\Psi_{p,\model}(\bX_i^{(b)};\wtheta ) \|_2^3  - \max_{1\leq b \leq B} \sum_{i=1}^{n_b} \|\Psi_{p,\model}(\bX_i^{(b)};\stheta ) \|_2^3 | \\ \leq  O_\mathbb{P}(n^{-\tau + \kappa/2}) \max_{1\leq b \leq B} \sum_{i=1}^{n_b} \|\Psi_{p,\model}(\bX_i^{(b)};\stheta ) \|_2^2.    
\end{multline*}
Therefore,
$$
\max_{1\leq b \leq B} \sum_{i=1}^{n_b} \|\Psi_{p,\model}(\bX_i^{(b)};\wtheta ) \|_2^3 = \max_{1\leq b \leq B} \sum_{i=1}^{n_b} \|\Psi_{p,\model}(\bX_i^{(b)};\stheta ) \|_2^3 + O_\mathbb{P}(n^{-\tau+2/q_0 + \kappa(5/2 + 2r_0/q_0)}).
$$
Noting that $n^{-\tau+2/q_0 + \kappa(5/2 + 2r_0/q_0)}=o(n^{3/q_0 + \kappa(3 + 3 r_0/q_0)})$, we have
$$
\max_{1\leq b \leq B} \sum_{i=1}^{n_b} \|\Psi_{p,\model}(\bX_i^{(b)};\wtheta ) \|_2^3 = O_\mathbb{P}(n^{3/q_0 + \kappa(3 + 3 r_0/q_0)}).
$$
%
In addition, using the definition of $\blambda_{\est}$, we have
$$
\lambda^{\dagger}_{\estall} \leq n_b^{-1/2}
\left(\bS_{\estall}^{-1} \bZ_{\estall}\right)^{\dagger}  + \bbeta^{\dagger}_{\estall},
$$
Since $\|\bbeta_{\est}\|_2=o_\mathbb{P}(\|\lambda_{\est}\|_2)$, then we have
$$
\lambda^{\dagger}_{\estall}(1+o_\mathbb{P}(1)) \leq  n_b^{-1/2}\left(\bS_{\estall}^{-1} \bZ_{\estall}\right)^{\dagger},
$$
This implies that
$$
\lambda^{\dagger}_{\estall}= O_\mathbb{P}(n^{-\rho/2})  Z^\dagger_{n,p,\wtheta} \left(\bS_{\estall}^{-1}\right)^{\dagger}.
$$
We have
$$
\left(\bS_{\estall}^{-1}\right)^{\dagger} \leq \|\bSigma_{\model,p}^{-1}\|_{sp} + \Gamma^\dagger_{n,p,\wtheta}.
$$
Assumptions~\ref{ass:main}-\ref{ass:cov} ensure that $\|\bSigma_{\model,p}^{-1}\|_{sp} =O(1)$ and Assumption~\ref{ass:main}-\ref{ass:growing3} ensures that $\vartheta_3>0$ leading by Lemma~\ref{lem:vectorrate} that $\Gamma^\dagger_{n,p,\wtheta}=o_\mathbb{P}(1)$. Hence,   we have
$$
\left(\bS_{\estall}^{-1}\right)^{\dagger}= O_\mathbb{P}(1).
$$
Hence, using the stochastic order of $ Z^\dagger_{n,p,\wtheta}$ stated  by Lemma~\ref{lem:vectorrate}, we have
$$
\lambda_{\estall}^{\dagger} = O_\mathbb{P}(n^{-\rho/2} [\ln^{1/2}n + n^{\kappa/2}]).
$$
Therefore, 
$$
 \eta^{\dagger}_{\estall} = O_\mathbb{P}(n^{-(3\rho - 6/q_0 - \kappa(6+6r_0/q_0))/2}[\ln^{3/2}n + n^{3\kappa/2}]) +  O_\mathbb{P}(n^{- \tau - \rho/2 +\kappa/2}[\ln^{3/2}n + n^{3\kappa/2}]).
$$

Hence, using Assumption~\ref{ass:main}-\ref{ass:growing3}, we have $n^{-(3\rho - 6/q_0 - \kappa(6+6r_0/q_0))/2}\ln^{3/2} n=o(n^{-\rho})$ and $n^{- \tau - \rho/2 +\kappa/2}\ln^{3/2} n=o(n^{-\tau})$ leading that
$\eta^{\dagger}_{\estall}= o_\mathbb{P}(1)$. And so,

\begin{equation} \label{eq:maxcv}
\left(Y_{\estall}  -   W_{\estall}\right)^{\dagger}= o_\mathbb{P}(1).
\end{equation}
We now need to control $\widetilde{W}_{\model,n,p}^{\dagger}$ with 

$$\widetilde{W}_{\model,n,p}^{\dagger}=\max_{1\leq b \leq B} 
| W_{\true} -  W_{\est}|.$$

Let $\widetilde{\bZ}_{\model,n,p,b}=\bZ_{\true} - \bZ_{\est}$. Note that, we have
$$
\widetilde{\bZ}_{\model,n,p}^{\dagger} \leq \max_{1\leq i\leq n} \|\Psi_{\model,p}(\bX_i;\wtheta) - \Psi_{\model,p}(\bX_i;\stheta) \|_2 \max_{1\leq b \leq B}n_b^{\rho/2}.
$$
Hence, we have
$$
\widetilde{\bZ}_{\model,n,p}^{\dagger}  = O_\mathbb{P}(n^{-\tau +  \rho/2 + \kappa/2})
$$
and
$$
\widetilde{W}_{\model,n,p,b} = \widetilde{\bZ}_{\model,n,p,b}^\top \bSigma_{\model,p}^{-1}(2\bZ_{\est} - \widetilde{\bZ}_{\model,n,p,b}).
$$
By using triangular inequality and Assumption~\ref{ass:main}-\ref{ass:cov}, there exists a positive constant $C$ such that
$$
\widetilde{W}_{\model,n,p}^{\dagger} \leq 
C\widetilde{\bZ}_{\model,n,p}^{\dagger} (2   Z^\dagger_{n,p,\wtheta} +\widetilde{\bZ}_{\model,n,p}^{\dagger}).
$$
Therefore, we have
$$
\widetilde{W}_{\model,n,p}^{\dagger} = O_\mathbb{P}(n^{-\tau + \rho/2+\kappa/2}[n^{\kappa} + \ln^{1/2}n]).
$$
Hence, by Assumptions~\ref{ass:main}-\ref{ass:cvtheta2}, \ref{ass:main}-\ref{ass:growing1} and \ref{ass:main}-\ref{ass:growing3}, we have
\begin{equation} \label{eq:maxcv2}
\widetilde{W}_{\model,n,p}^{\dagger} = o_\mathbb{P}(1).
\end{equation}
Combining \eqref{eq:maxcv} and \eqref{eq:maxcv2} provides \eqref{Appendix A: eq1}. 
Noting that $W_{\trueall}=(W_{\trueall,1},\ldots,W_{\trueall,B})^\top$ is a continuous random variable, as a direct consequence of the convergence in probability, this implies that
$$
\lim_{n\to \infty} \| F_{Y_{\estall}} - F_{W_{\trueall}}\|_\infty = 0,
$$
where $F_{Y_{\estall}} $ and $F_{W_{\trueall}}$ are  the cumulative distribution function of the $B$-dimensional vectors $Y_{\estall}=(Y_{\estall,1},\ldots,Y_{\estall,B})^\top$ and $W_{\trueall}$ respectively. To conclude the proof it suffices to show that
\begin{equation}\label{eq:obj2}
\lim_{n\to \infty} \| F_{W_{\trueall}} - \prod_{b=1}^B F_{\mathcal{X}^2_p}\|_\infty = 0,
\end{equation}
where  $F_{\mathcal{X}^2_p}$ denotes  the cumulative distribution function of a chi-square random variable with $p$ degrees of freedom.

Let $\overline{\bZ}_{\true}=\bSigma_{\model,p}^{-1/2}\bZ_{\true}$ be the $p$-variate vector with non-correlated components and having $F_{\overline{\bZ}_{\true}}$ as cumulative distribution function. We have for any $t\in\mathbb{R}$, 
$$
F_{W_{\true}}(t) - F_{\mathcal{X}^2_p}(t) = \int_{\|\overline{\bZ}_{\true}\|_2^2\leq t} dF_{\overline{\bZ}_{\true}} -  \int_{\|\overline{\bZ}_{\true}\|_2^2\leq t} d\Phi_{p},
$$
where $F_{\mathcal{X}^2_p}$ is the cumulative function of a chi-square random variable with $p$ degrees of freedom. 
Note that, for any $t\in\mathbb{R}^+$, we have
$$
\left| \int_{\|\overline{\bZ}_{\true}\|_2^2\leq t} dF_{\overline{\bZ}_{\true}} -  \int_{\|\overline{\bZ}_{\true}\|_2^2\leq t} d\Phi_{p}\right| \leq \Delta_n,
$$
with
$$\Delta_n= \sup_{A \in \mathcal{C}} |\mathbb{P}(\overline{\bZ}_{\true} \in A) - \bnu(A)|,$$
 where $\mathcal{C}$ is the class of convex subsets of $\mathbb{R}^p$ and $\bnu$ is the standard $p$ dimensional normal distribution. From \citet{bentkus2003dependence}, we have
 $$
 \Delta_n \leq 400 p^{1/4} \mathbb{E}[\|\bSigma_{\model,p}^{-1/2}\Psi_{\model,p}(\bX;\stheta) \|_2^3] n_b^{-1/2}.
 $$
 Noting that by Assumptions~\ref{ass:main}-\ref{ass:cov}, we have $\mathbb{E}[\|\bSigma_{\model,p}^{-1/2}\Psi_{\model,p}(\bX;\stheta) \|_2^3] =\mathbb{E}[\|\Psi_{\model,p}(\bX;\stheta) \|_2^3]$.  Hence, using \eqref{eq:momentS}, we have
 \begin{equation}\label{eq:ratechi2}
\| F_{W_{\true}} - F_{\mathcal{X}^2_p}\|_\infty = O(  n^{-\rho/2 + \kappa\left(\frac{7}{4} + \frac{3r_0}{q_0}\right)}).
 \end{equation}
 
Note that the vector $W_{\true}$ consists of independent components, owing to the independence of the observations and the fact that it is evaluated at the true parameter value rather than at an estimator. The proof is concluded by noting that Assumption~\ref{ass:main}.\ref{ass:growing3} ensures that:
$n^{-\rho/2 + \kappa(7/4 + 3r_0/q_0)} = o(1)$,
and thus \eqref{eq:obj2} holds true.

\end{proof}

\begin{proof}[Proof of Corollary~\ref{cor:gof_stats}:]
The results for $\bar{Y}_{\estall}$ and $Y_{\estall}^{\dagger}$ follow directly from the Continuous Mapping Theorem applied to the vector $Y_{\estall}$, whose components are asymptotically i.i.d. $\chi^2_p$ according to Theorem~\ref{thm:niveau}. For the EDF-based statistics ($Y_{\estall}$-KS, $Y_{\estall}$-CV, and $Y_{\estall}$-AD), the convergence follows from the fact that the target distribution $F_{\chi^2_p}$ is fully specified and independent of any estimated parameters, allowing the use of the standard asymptotic theory for empirical processes.
\end{proof}

\begin{proof}[Proof of Theorem~\ref{thm:alternative}:]
We start by controlling of the maximum term $V_{\est}^{\alt} = \max_{1\leq i \leq n_b} \|\Psi_{\model,p}^{\alt}(\bX_i^{(b)};\wtheta) \|_2$, using arguments similar to those employed in the proof of Lemma~\ref{lem:controlMax}. Specifically, these arguments rely on the fact that replacing $\theta$ with $\wtheta$ results in a negligible remainder and on the existence of $q_0$-th order moments. Since the transition to a local alternative does not affect these two properties, we obtain:
\begin{equation}\label{eq:maxdev}
    V_{\est}^{\alt} = o_\mathbb{P}(n^{\rho/2} p^{-1}).
\end{equation}
We consider the matrices $\bSigma_{\model,p}^{\alt}= \mathbb{E}[\Psi_{\model,p}^{\alt}(\bX;\stheta)\Psi_{\model,p}^{\alt}(\bX;\stheta)^\top]$, 
$\bS_{\any}^{\alt}=\frac{1}{n_b}\sum_{i=1}^{n_b} \Psi_{\model,p}^{\alt}(\bX_i^{(b)};\btheta)\Psi_{\model,p}^{\alt}(\bX_i^{(b)};\btheta)^\top$ 
and
$\bGamma_{\any}^{\alt}=(\bS_{\any}^{\alt})^{-1} -  (\bSigma_{\model,p}^{\alt})^{-1}$. 
Following the same arguments as those employed in the proof of  Lemma~\ref{lem:covmatrix}, we have $\|\bS_{\est}^{\alt} - \bSigma_{\model,p}^{\alt}\|_{sp}=O_\mathbb{P}(n^{-\vartheta_1})$. In addition, we have
$
\bSigma_{\model,p}^{\alt}=\bSigma_{\model,p} + \frac{1}{n_b} \tau^2 \Gamma \Gamma^\top$. 
Since the elements of the $p$-dimensional vector $\Gamma$ are bounded uniformly on $p$, we have that $\|\bSigma_{\model,p}^{\alt}-\bSigma_{\model,p} \|_{sp}=O(pn_b^{-1})$. From Assumptions~\ref{ass:main}, $pn_b^{-1}=o(n^{-\vartheta_1})$ and thus
\begin{equation}\label{eq:spdev}
\|\bS_{\any}^{\alt} - \bSigma_{\model,p}\|_{sp}=o_\mathbb{P}(n^{-\vartheta_1}).
\end{equation}
Following the same arguments as in the proof of Lemma~\ref{lem:covmatrix} to control the spectral norm of $\bGamma_{\est}$, we obtain that
\begin{equation}\label{eq:spdev}
\|\bGamma_{\est}^{\alt}\|_{sp}=O_\mathbb{P}(n^{-\vartheta_1}),
\end{equation}
where $\bGamma_{\est}^{\alt}=(\bS_{\any}^{\alt})^{-1} - \bSigma_{\model,p}^{-1}$.  
Let $\blambda_{\est}^{\alt}$ be the Lagrange multipliers that satisfies the empirical version of the moment condition, leading that
$$
\sum_{i=1}^{n_b}  [n_b(1 + (\blambda_{\est}^{\alt})^\top \Psi_{\model,p}^{\alt}(\bX_i^{(b)};\wtheta))]^{-1} \Psi_{\model,p}^{\alt}(\bX_i^{(b)};\wtheta)=\bzero_p.
$$
To bound the magnitude of the Lagrange multipliers, we define $\blambda_{\est}^{\alt}=\|\blambda_{\est}^{\alt}\|_2 \bnu $ where $\bnu$ is a unit vector of $\mathbb{R}^p$. 
With the arguments used in the proof of Lemma~\ref{lem:lagrange}, we can obtain
$$
 \|\blambda_{\est}^{\alt}\|_2 \left(\bnu^{\top}     \bS_{\est}^{\alt}\bnu - n_b^{-1/2} \bnu^\top  \bZ_{\est}^{\alt}V_{\est}^{\alt} \right) \leq  n_b^{-1/2} \bnu^\top  \bZ_{\est}^{\alt}.
$$
Since $\bZ_{\est}-\bZ_{\est}^{\alt}=\tau\Gamma$, we have $\|\bZ_{\est}-\bZ_{\est}^{\alt}\|_2=O_\mathbb{P}(p^{1/2})$, hence using \eqref{eq:step1b}, we have
$ n_b^{-1/2} \bnu^\top  \bZ_{\est}^{\alt} = O_\mathbb{P}( n^{-\rho/2} p^{1/2})$.
From \eqref{eq:maxdev}, we have $ n_b^{-1/2} \bnu^\top  \bZ_{\est}^{\alt}V_{\est}^{\alt} = o_\mathbb{P}(p^{-1/2})$.
We have by triangular inequality
$
\bnu^\top  \bS_{\est} \bnu \leq \bnu^\top   \bSigma_{\model,p} \bnu + |\bnu^\top  [\bS_{\est}- \bSigma_{\model,p} ]\bnu|$. By Assumption~\ref{ass:main}-\ref{ass:cov}, we have $\bnu^\top   \bSigma_{\model,p} \bnu =O(1)$. Hence, using \eqref{eq:spdev}  and using Assumption~\ref{ass:main}-\ref{ass:growing3} ensuring that $n^{-\vartheta_1}=o(1)$, we obtain that
$\bnu^\top  \bS_{\est}^{\alt} \bnu  = O_\mathbb{P}(1)$
and thus 
\begin{equation}\label{eq:lamdev}
\|\blambda_{\est}^{\alt}\|_2 =O_\mathbb{P}( n^{-\rho/2} p^{1/2}).
\end{equation}
We define
$$\bzeta_{\est}^{\alt} = \frac{1}{n_b} \sum_{i=1}^n \Psi_{\model,p}^{\alt}(\bX_i^{(b)};\wtheta) \frac{(U_{\est,i}^{\alt})^2}{1+U_{\est,i}^{\alt}},
$$
with $U_{\est,i}^{\alt}=(\blambda_{\est}^{\alt})^\top_2V_{\est}^{\alt}$. Using \eqref{eq:maxdev} and \eqref{eq:lamdev}, we have $
 \max_{1\leq i\leq n_b} |1+U_{\est,i} |^{-1} = O_\mathbb{P}(1)$. In addition, we have 
$\frac{1}{n_b}\sum_{i=1}^{n_b}  \|\Psi_{\model,p}(\bX_i^{(b)};\stheta)\|_2^3=O_\mathbb{P}(n^{\kappa(3/2+3r_0/q_0)})$ and thus using the counter-part of \eqref{eq:zeta} that consider the deviation $\altnop$, we have
$$\left\|\bzeta_{\est}^{\alt} \right\|_2 =  O_\mathbb{P}(n^{-\vartheta_2}).$$
We have 
\begin{equation}\label{eq:dlzetaalt}
n_b^{-1/2} \bZ_{\est}^{\alt} - \bS_{\est}^{\alt} \blambda_{\est}^{\alt} + \bzeta_{\est}^{\alt} =  \bzero_p.
\end{equation}
We have shown that  $\sigma_{1}^{-1}(\bS_{\est})=O_\mathbb{P}(1)$. Hence, 
$$\blambda_{\est}^{\alt}  = \frac{1}{n_b^{1/2}}\left(\bS_{\est}^{\alt}\right)^{-1} \bZ_{\est}^{\alt}  + \bbeta_{\est}^{\alt}.$$
 Since we have  $\left\|\bbeta_{\est}^{\alt}\right\|_2\leq \sigma_{1}^{-1}(\bS_{\est}^{\alt})\left\|\bzeta_{\est}^{\alt} \right\|_2 $, 
 then
 \begin{equation}\label{eq:betaalt}
 \left\|\bbeta_{\est}^{\alt}\right\|_2 = O_\mathbb{P}(n^{-\vartheta_2}).     
 \end{equation}
Using \eqref{eq:maxdev}-\eqref{eq:betaalt}, the magnitude of the deviation, and the same reasoning that those used for the proof of Theorem~\ref{thm:niveau}, we obtain that
$
\left(Y_{\estall}^{\alt}  -   W_{\estall}^{\alt}\right)^{\dagger}= o_\mathbb{P}(1),
$
and
$
\left(W_{\trueb}^{\alt}  -   W_{\estall}^{\alt}\right)^{\dagger}= o_\mathbb{P}(1),
$
where $W_{\any}^{\alt}=\left(\bZ_{\any}^{\alt}\right)^\top\left(\bSigma_{\model,p}^{\alt}\right)^{-1}  \bZ_{\any}^{\alt}$ with $\bSigma_{\model,p}^{\alt} = \mathbb{E}[\Psi_{\model,p}^{\alt}(\bX_i^{(b)};\stheta)\Psi_{\model,p}^{\alt}(\bX_i^{(b)};\stheta)^\top] $, leading that
$$
\left(Y_{\estall}^{\alt}  -   W_{\trueb}^{\alt}\right)^{\dagger}= o_\mathbb{P}(1).
$$
Since $\bSigma_{\model,p}^{\alt}=\bSigma_{\model,p} + O_\mathbb{P}(n_b^{-1})$, $\bZ_{\true}^{\alt}=\bZ_{\true} + \Upsilon$,  $\mathbb{E}[\bZ_{\true}]=0$ and $\bZ_{\true}=O_\mathbb{P}(n_b^{-1/2})$, we have
$$
W_{\true}^{\alt}=W_{\true} + \Upsilon^\top\bSigma_{\model,p}^{-1}  \Upsilon + o_\mathbb{P}(1).
$$
Using \eqref{eq:obj2}, we have
$$
\lim_{n\to \infty} \| F_{Y_{\estall}} - \prod_{b=1}^B F_{\mathcal{X}^2_{p}(\Upsilon^\top\bSigma_{\model,p}^{-1}  \Upsilon)}\|_\infty = 0,
$$
where $F_{\mathcal{X}^2_{p}(\Upsilon^\top\bSigma_{\model,p}^{-1}  \Upsilon)}$ denotes the cumulative distribution of a non-central chi-square random variable with $p$ degrees of freedom and non-central parameter equal to $\Upsilon^\top\bSigma_{\model,p}^{-1}  \Upsilon$.
\end{proof}

\section{Proofs of technical results}\label{B.2}

\begin{proof}[Proof of Lemma~\ref{lem:controlMax}:]
Triangular inequality implies that
$$
V_{\est} \leq V_{\true} + |V_{\est} - V_{\true}|.
$$ 
Note that 
$$|V_{\est} - V_{\true}| \leq \max_{1\leq i\leq n} \|\Psi_{\model,p}(\bX_i;\wtheta) - \Psi_{\model,p}(\bX_i;\stheta) \|_2.$$
Hence, Assumptions~\ref{ass:main}-\ref{ass:cvtheta2} implies that $$V_{\est}-V_{\true}=O_\mathbb{P}(n^{-\tau} p^{1/2}).$$ 
To control $V_{\true}$, we follow the idea of  \citet[Lemma 4.1]{Hjort2009AOS}. Using the union bound, we have for any $\varepsilon>0$
$$
\mathbb{P}(V_{\true}\geq \varepsilon ) \leq \sum_{i=1}^{n_b} \mathbb{P}(\|\Psi_{\model,p}(\bX_i^{(b)};\stheta)\|_2 \geq  \varepsilon).
$$
Hence, using Markov's inequality, we have
$$
\mathbb{P}(V_{\true}\geq \varepsilon ) \leq \frac{n_b}{\varepsilon^{q_0}}  \mathbb{E}[\|\Psi_{\model,p}(\bX_i^{(b)};\stheta)\|_2^{q_0}].
$$
Since Assumptions~\ref{ass:main}-\ref{ass:growing2} implies \eqref{eq:moment}, we have
$$
\mathbb{P}(pn^{-\rho/2} V_{\true}\geq \varepsilon ) \leq\frac{n_b}{\varepsilon^{q_0} n^{\rho q_0/2}}p^{3q_0/2+r_0}\tilde{C}.
$$
Using Assumptions~\ref{ass:main}-\ref{ass:cvtheta2} and \ref{ass:main}-\ref{ass:growing3}, there exists a positive constant $C_1$ and a constant $v_4=\kappa(3q_0/2+r_0)+   \rho(1-q_0/2)$ such that
$$
\mathbb{P}(pn^{-\rho/2} V_{\true}\geq \varepsilon ) \leq \frac{C_1}{\varepsilon^{q_0}}n^{v_4}.
$$

From Assumption~\ref{ass:main}-\ref{ass:growing2} and \ref{ass:main}-\ref{ass:growing3}, we have 
\begin{equation*}
v_4 <\frac{\rho}{6}(3q_0/2+r_0)+\rho(1-q_0/2)=\rho(1-q_0/4+r_0/6)\leq  \rho(1-q_0/4) 
\end{equation*}
Hence, since $q_0 \geq 4$, we have $v_4<0$ leading that for any $\varepsilon>0$,
$$
\lim_{n\to\infty} \mathbb{P}(pn^{-\rho/2} V_{\true}\geq \varepsilon ) =0.
$$
Therefore, for any $b$, $$V_{\true}=o_\mathbb{P}(n^{\rho/2}p^{-1}),$$
leading that
$$
V_{\est}=  o_\mathbb{P}(n^{\rho/2}p^{-1}).
$$
\end{proof}

\begin{proof}[Proof of Lemma~\ref{lem:covmatrix}] 
    Let $\upsilon_{\any} = \| \bSigma_{\model,p}  -  \bS_{\any} \|_{\max}$. Since $\bSigma_{\model,p}  -  \bS_{\any} $ is a square matrix of size $p\times p$, we have
$$ \left\|\bSigma_{\model,p}-\bS_{\any} \right\|_{sp} \leq p\upsilon_{\any}.$$ 
Let 
$$
\widetilde{\upsilon}_{n,b} = |\upsilon_{\est} - \upsilon_{\true}|.
$$
Noting that by triangular inequality
$$
\upsilon_{\est} \leq \upsilon_{\true} + \widetilde{\upsilon}_{n,b},
$$
we have
$$
\left\|\bSigma_{\model,p}-\bS_{\est} \right\|_{sp} \leq p\upsilon_{\true} + p\widetilde{\upsilon}_{n,b}.
$$
We have
$$
\widetilde{\upsilon}_{n,b}  \leq \max_{p,j} \max_{1\leq i \leq n}| \psi_{\model,\wtheta,\varphi_{p,j}}(\bX_i) -  \psi_{\model,\stheta,\varphi_{p,j}}(\bX_i)|,$$
leading by Assumptions~\ref{ass:main}-\ref{ass:cvtheta2} and \ref{ass:main}-\ref{ass:growing3} that $p\widetilde{\upsilon}_{n,b}=O_\mathbb{P}(n^{-\tau+\kappa})$. Since all the components of $\Psi_{\model,p}(\bX_i^{(b)};\stheta)$ admit a $q_0$-th order moments by Assumption~\ref{ass:main}-\ref{ass:growing2}, then from \citet[Lemma 4.4]{Hjort2009AOS}, there exists a positive constant $C_2$ such that for any $\varepsilon>0$, we have
$$
\mathbb{P}(\upsilon_{\true}\geq \varepsilon) \leq \frac{C_2 p^2}{\varepsilon^{q_0} n_b^{q_0/2}}a_{\true,q_0}^2.
$$
where $a_{\any,q} = p^{-1} \sum_{j=1}^p \mathbb{E} |\psi_{\model,\btheta_{\model},\varphi_{p,j}}(\bX_i^{(b)})|^q$. By  Assumption~\ref{ass:main}-\ref{ass:growing2}, we have $a_{\true,q_0}\leq \tilde Cp^{r_0}$. Hence, there exists a positive constant $C_3$, such that for any $\varepsilon>0$,
\begin{equation} \label{eq:boundupsilon}
\mathbb{P}(\upsilon_{\true}\geq \varepsilon) \leq \frac{C_3}{\varepsilon^{q_0}}  n^{-q_0(\rho/2 -\kappa(2(1+r_0)/q_0))}.
\end{equation}
Hence, from Assumption~\ref{ass:main}-\ref{ass:growing3}, we have
\begin{equation} \label{eq:bounddiffsigma}
\mathbb{P}(p\upsilon_{\true} \geq \varepsilon) \leq \frac{C_3}{\varepsilon^{q_0}}  n^{-q_0\vartheta_1},
\end{equation}
where $\vartheta_1=\rho/2 -\kappa(1+2(1+r_0)/q_0)>0$, and so
$$
\mathbb{P}(n^{\vartheta_1}p\upsilon_{\true}\geq \varepsilon) \leq \frac{C_2}{\varepsilon^{q_0}}.
$$
Therefore, we have
$$
\left\|\bSigma_{\model,p}-\bS_{\est} \right\|_{sp} = O_{\mathbb{P}}(  n^{-\vartheta_1}) + O_{\mathbb{P}}(n^{-\tau + \kappa}).
$$
Since by Assumption~\ref{ass:main}-\ref{ass:growing1}, we have $\rho/2<\tau$, then $n^{-\tau + \kappa}=o(n^{-\vartheta_1})$. This implies that 
$$
\left\|\bSigma_{\model,p}-\bS_{\est} \right\|_{sp} = O_{\mathbb{P}}(  n^{-\vartheta_1}).
$$

 To control the spectral norm of $\bGamma_{\est}=\bS^{-1}_{\est}-\bSigma^{-1}_{\model,p}$, we use the following decomposition $\bGamma_{\est}= \bS^{-1}_{\est}(\bSigma_{\model,p}-\bS_{\est})\bSigma^{-1}_{\model,p}$ and so we have the inequality
 $$
\| \bGamma_{\est} \|_{sp} \leq \frac{\| \bSigma_{\model,p}-\bS_{\est} \|_{sp}}{\sigma_1(\bSigma_{\model,p}) \sigma_{1}(\bS_{\est})}.
 $$
Weyl's inequality implies that
$$|\sigma_1(\bSigma_{\model,p}) - \sigma_{1}(\bS_{\est})|\leq  \left\|\bSigma_{\model,p}-\bS_{\est} \right\|_{sp}.$$
By Assumption~\ref{ass:main}-\ref{ass:cov}, $\sigma_1(\bSigma_{\model,p})=O(1)$. In addition, since $\vartheta_1>0$, $\sigma_1(\bS_{\true})$ converges in probability to the smallest singular value of $\bSigma_{\model,p}$ leading that $\sigma_1(\bS_{\est})=O_\mathbb{P}(1)$ where the order holds uniformly on $p$. Therefore, we have
 $$
\| \bGamma_{\est} \|_{sp}=  O_{\mathbb{P}}(  n^{-\vartheta_1}).
 $$
\end{proof}

\begin{proof}[Proof of Lemma~\ref{lem:lagrange}:]
This proof extends the Lagrange multipliers proof provided by \citet[page~221]{owen2001empirical} to the case of growing dimension.  First, note that maximizing the empirical likelihood with respect to the weights implies that
$$
\xi_{\any,i} = [n_b(1 + \blamany^\top \Psi_{\model,p}(\bX_i^{(b)};\btheta))]^{-1}.
$$ 
The Lagrange multipliers satisfies the empirical counter-part of the moment condition $\mathbb{E}[\Psi_{\model,p}(\bX_i^{(b)};\btheta)]=\bzero_p$, leading that
\begin{equation}\label{eq:weights}
\sum_{i=1}^{n_b}  [n_b(1 + \blamany^\top \Psi_{\model,p}(\bX_i^{(b)};\btheta))]^{-1} \Psi_{\model,p}(\bX_i^{(b)};\btheta)=\bzero_p.
\end{equation}
To bound the magnitude of the Lagrange multipliers, we define $\blamany=\|\blamany\|_2 \bnu $ where $\bnu$ is a unit vector of $\mathbb{R}^p$. Let $U_{\any,i}=\Psi_{\model,p}(\bX_i^{(b)};\btheta)^\top\blamany$. Noting that $(1+U_{\any,i})^{-1} = 1 - U_{\any,i}/(1+U_{\any,i})$, from \eqref{eq:weights}, we have
\begin{equation}\label{eq:weights2}
\frac{1}{n_b} \sum_{i=1}^n \frac{U_{\any,i}}{1+U_{\any,i}} \Psi_{\model,p}(\bX_i^{(b)};\btheta) = \frac{1}{n_b} \sum_{i=1}^n \Psi_{\model,p}(\bX_i^{(b)};\btheta).
\end{equation}
Let $\bS^{\bw}_{\any}$ the weighted empirical covariance matrix of $\Psi_p(\bX_i^{(b)};\btheta  )$ defined by
$$
\bS^{\bw}_{\any}=\frac{1}{n_b} \sum_{i=1}^{n_b} \frac{1}{1+U_{\any,i}}  \Psi_{\model,p}(\bX_i^{(b)};\btheta) \Psi_{\model,p}(\bX_i^{(b)};\btheta)^\top.
$$
Then, multiplying both sides of  \eqref{eq:weights2} by $\bnu^\top$, we have
\begin{equation} \label{eq:rellambda}
  \bnu^\top  \bS^{\bw}_{\any} \bnu \|\blamany\|_2 = n_b^{-1/2} \bnu^\top  \bZany.
\end{equation}
Define the unweighted empirical covariance matrix $\bS_{\any}$ by
$$\bS_{\any}=\frac{1}{n_b} \sum_{i=1}^{n_b} \Psi_{\model,p}(\bX_i^{(b)};\btheta)\Psi_{\model,p}(\bX_i^{(b)};\btheta)^\top.$$
Since all the weights $\xi_{\any,i}$ are strictly positive, then
$$\bnu^\top \bS_{\any} \bnu \leq \bnu^\top \bS^{\bw}_{\any}\bnu (1+\max_{1\leq i \leq n_b} U_{\any,i}).$$
Noting that  $\max_{1\leq i \leq n_b} |U_{\any,i}| \leq \|\blamany\|_2 V_{\any}$, we have for any $\btheta$
\begin{equation*}\label{eq:orderlambdap}
    \|\blamany\|_2 \bnu^{\top}     \bS_{\any}  \bnu   
   \leq  \|\blamany\|_2 \bnu^{\top}  \bS^{\bw}_{\any}\bnu (1+   \|\blamany\|_2 V_{\any} ).
\end{equation*}
Using \eqref{eq:rellambda} to replace $ \|\blamany\|_2 \bnu^{\top}  \bS^{\bw}_{\any}\bnu $ in  the previous inequality then evaluating the resulting inequality at $\btheta=\wtheta$    gives
\begin{equation}\label{eq:step1}
 \|\blambda_{\est}\|_2 \left(\bnu^{\top}     \bS_{\est}\bnu - n_b^{-1/2} \bnu^\top  \bZ_{\est}V_{\est} \right) \leq  n_b^{-1/2} \bnu^\top  \bZ_{\est}.
\end{equation}
Let $\widetilde{\bZ}_{\model,n,p,b}= \bZ_{\true} - \bZ_{\est}$, triangular inequality implies
\begin{equation*}
\|  \bZ_{\est}\|_2\leq  \| \bZ_{\true} \|_2+ \|\widetilde{\bZ}_{\model,n,p,b}\|_2.
\end{equation*}
Noting that by Assumption~\ref{ass:main}-\ref{ass:cvtheta2}, 
$\widetilde{\bZ}_{\model,n,p}^{\dagger}=O_\mathbb{P}(n^{-\tau + \rho/2}p^{1/2})$. Since, $\rho/2<\tau$, we have
$$\widetilde{\bZ}_{\model,n,p}^{\dagger}=o_\mathbb{P}(p^{1/2}).$$
Note that 
$$
 \|\bZ_{\true}\|_2^2 = \sum_{j=1}^p  \frac{1}{n_b} \sum_{i=1}^{n_b} \sum_{i'=1}^{n_b}  \psi_{\model,\btheta_0,\varphi_{p,j}}(\bX_i^{(b)}) \psi_{\model,\btheta_0,\varphi_{p,j}}(\bX_{i'}^{(b)}).
$$
Since the observations are independent and $\psi_{\model,\btheta_0,\varphi_{p,j}}(\bX_i^{(b)})$ is centered then, 
$$\mathbb{E}[\|\bZ_{\true}\|_2^2]=\text{trace}(\bSigma_{\model,p}).$$
Since Assumption~\ref{ass:main}-\ref{ass:cov} implies that $\text{trace}(\bSigma_{\model,p})=O(p)$, Markov's inequality with second order moment implies that
$
\| \bZ_{\true}\|_2=O_\mathbb{P}(p^{1/2})
$ 
and thus
\begin{equation} \label{eq:step1b}
\| \bZ_{\est}\|_2=O_\mathbb{P}(p^{1/2}).
\end{equation}
Therefore, 
\begin{equation} \label{eq:step2}
 n_b^{-1/2} \bnu^\top  \bZ_{\est} = O_\mathbb{P}( n^{-\rho/2} p^{1/2}).
\end{equation}
Using Lemma~\ref{lem:controlMax} to control  $V_{\est}$, we have
\begin{equation}\label{eq:step3}
 n_b^{-1/2} \bnu^\top  \bZ_{\est}V_{\est} = o_\mathbb{P}(p^{-1/2}).
\end{equation}
We have by triangular inequality
$$
\bnu^\top  \bS_{\est} \bnu \leq \bnu^\top   \bSigma_{\model,p} \bnu + |\bnu^\top  [\bS_{\est}- \bSigma_{\model,p} ]\bnu|.
$$
By Assumption~\ref{ass:main}-\ref{ass:cov}, we have $\bnu^\top   \bSigma_{\model,p} \bnu =O(1)$. In addition, we have 
\begin{align*}
|\bnu^\top  [\bS_{\est}- \bSigma_{\model,p} ]\bnu| &\leq \|\bS_{\est}- \bSigma_{\model,p} \|_{sp} \|\bnu\|_2^2 \\
&=\|\bS_{\est}- \bSigma_{\model,p} \|_{sp}.
\end{align*}
Hence, using Lemma~\ref{lem:covmatrix} for the order of this spectral norm and using Assumption~\ref{ass:main}-\ref{ass:growing3} ensuring that $n^{-\vartheta_1}=o(1)$, we obtain that
\begin{equation} \label{eq:step4}
\bnu^\top  \bS_{\est} \bnu  = O_\mathbb{P}(1).
\end{equation}
Starting from \eqref{eq:step1} and using \eqref{eq:step2}, \eqref{eq:step3} and \eqref{eq:step4}, we have
$$
 \|\blambda_{\est}\|_2 ( O_\mathbb{P}(1) - o_\mathbb{P}(1) ) = O_\mathbb{P}( n^{-\rho} p^{1/2}),
$$
and thus $$\|\blambda_{\est}\|_2 =O_\mathbb{P}( n^{-\rho/2} p^{1/2}).$$
We now give the establish the asymptotic expansion of the Lagrange multipliers. Hence, we define 
\begin{equation}\label{eq:defzeta}
\bzeta_{\any} = \frac{1}{n_b} \sum_{i=1}^n \Psi_{\model,p}(\bX_i^{(b)};\btheta) \frac{U_{\any,i}^2}{1+U_{\any,i}}.
\end{equation}
 Using the triangle inequality, we have
\begin{align}
   \left\|\bzeta_{\any} \right\|_2 \leq & \frac{1}{n_b} \sum_{i=1}^{n_b} \left\|\Psi_{\model,p}(\bX_i^{(b)};\btheta) \frac{U_{\any,i}^2}{1+U_{\any,i}}\right\|_2 \nonumber\\
  \leq &  \|\blamany\|_2^2 \left(\max_{1\leq i \leq n_b}   |1+U_{\any,i}|^{-1}\right)   \frac{1}{n_b} \sum_{i=1}^{n_b}  \|\Psi_{\model,p}(\bX_i^{(b)};\btheta)\|_2^3. \label{eq:zeta}
\end{align}
Since for any $\btheta$, we have $$\max_{1\leq i \leq n_b} |U_{\any,i}| \leq \| \blambda_{\any}\|_2V_{\any},$$ then, we have 
$$
\max_{1\leq i \leq n_b} |U_{\est,i}|  =o_\mathbb{P}(p^{-1/2}).
$$
Taylor expansion of $(1+s)^{-1}$ around $s=0$ implies that
$$
\max_{1\leq i \leq n_b}|1+U_{\est,i} |^{-1} \leq \max_{1\leq i \leq n_b} |1+(1+o(1))\max_{1\leq i \leq n_b}U_{\est,i}|
$$
and  hence $$
 \max_{1\leq i\leq n_b} |1+U_{\est,i} |^{-1} = O_\mathbb{P}(1).
$$
Assumption~\ref{ass:main}-\ref{ass:growing2} implies \eqref{eq:moment}, so by Hölder's inequality    for any integer $s$ such that $s\leq q_0$, we have
$$
\mathbb{E} \|\Psi_{\model,p}(\bX_i^{(b)};\stheta)\|_2^{s}=O(p^{s/2+sr_0/q_0}).
$$
In addition, Minkowski's inequality implies that
$$
 \|\Psi_{\model,p}(\bX_i^{(b)};\wtheta)\|_2^{s} \leq 2^{s/2} \left(\|\Psi_{\model,p}(\bX_i^{(b)};\stheta)\|_2^s + \|\Psi_{\model,p}(\bX_i^{(b)};\wtheta) - \Psi_{\model,p}(\bX_i^{(b)};\stheta) \|_2^s \right).
$$
Since $\max_{1\leq i\leq n} \|\Psi_{\model,p}(\bX_i;\wtheta) - \Psi_{\model,p}(\bX_i;\stheta) \|_2=O_\mathbb{P}(n^{-\tau}p^{1/2})$ and $n^{-\tau}p^{1/2}=o(p^{3/2+3r_0/q_0})$,  we have
$$\mathbb{E} \|\Psi_{\model,p}(\bX_1;\wtheta)\|_2^{3}=O(p^{3/2+3r_0/q_0}).$$
Law of Large Number and Assumptions~\ref{ass:main}-\ref{ass:growing2} imply that 
$$\frac{1}{n_b}\sum_{i=1}^{n_b}  \|\Psi_{\model,p}(\bX_i^{(b)};\stheta)\|_2^3=O_\mathbb{P}(n^{\kappa(3/2+3r_0/q_0)}).$$ 
Let $\vartheta_2=\rho -\kappa(5/2+3r_0/q_0) >0$ from Assumptions~\ref{ass:main}-\ref{ass:growing3}, then  from \eqref{eq:zeta}, we have
$$\left\|\bzeta_{\est} \right\|_2 =  O_\mathbb{P}(n^{-\vartheta_2}).$$
Noting that $(1+U_{\any,i})^{-1} = 1 - U_{\any,i} +  U_{\any,i}^2/(1+U_{\any,i})$ and that the Lagrange multipliers satisfies the empirical counter-part of the moment condition $\mathbb{E}[\Psi_{\model,p}(\bX_i^{(b)};\btheta)]=\bzero_p$ (see \eqref{eq:weights}), we have for any $\btheta$
\begin{equation}\label{eq:dlzeta}
n_b^{-1/2} \bZany - \bS_{\any} \blamany + \bzeta_{\any} =  \bzero_p.
\end{equation}
 For any $\btheta$ such that  $\bS_{\any}$ is invertible, define  $$\bbeta_{\any}=\bS_{\any}^{-1}\bzeta_{\any}.$$ 
We have shown that  $\sigma_{1}^{-1}(\bS_{\est})=O_\mathbb{P}(1)$. Hence, considering \eqref{eq:dlzeta} at $\btheta=\wtheta$ and multiplying by the inverse of $\bS_{\est}$ yields
$$\blambda_{\est}  = \frac{1}{n_b^{1/2}}\left(\bS_{\est}\right)^{-1} \bZ_{\est}  + \bbeta_{\est}.$$
 Since we have  $$\left\|\bbeta_{\est}\right\|_2\leq \sigma_{1}^{-1}(\bS_{\est})\left\|\bzeta_{\est} \right\|_2, $$
 then
 $$
 \left\|\bbeta_{\est}\right\|_2 = O_\mathbb{P}(n^{-\vartheta_2}).
 $$
 
\end{proof}

\begin{proof}[Proof of Lemma~\ref{lem:vectorrate}:]
Since we have
$$
|Z^\dagger_{n,p,\wtheta} - Z^\dagger_{n,p,\stheta}|\leq \max_{1\leq b \leq B} \left[n_b^{1/2} \max_{1\leq i \leq n_b} \|\Psi_{\model,p}(\bX_i^{(b)};\stheta)-\Psi_{\model,p}(\bX_i^{(b)};\wtheta) \|_2\right], 
$$
then by Assumption~\ref{ass:main}-\ref{ass:cvtheta2}, \ref{ass:main}-\ref{ass:growing1}, and \ref{ass:main}-\ref{ass:growing3}, we have
\begin{equation}\label{eq:oderapproxZ}
Z^\dagger_{n,p,\wtheta}= Z^\dagger_{n,p,\stheta} + O_\mathbb{P}(n^{-\tau+\rho/2+\kappa/2}).\end{equation}
To control  $Z^\dagger_{n,p,\stheta}$, we define for any $\btheta$,   $\bG_{\any} = \bSigma_{\model,p}^{-1/2}\bZ_{\any}$. We have
  $$\|\bZ_{\true}\|_2 \leq \| \bSigma_{\model,p}^{1/2}\|_{sp} \|\bG_{\true}\|_2,$$
  leading by Assumption \ref{ass:main}-\ref{ass:cov} that
    $$\|\bZ_{\true}\|_2 =O_\mathbb{P}(\|\bG_{\true}\|_2),$$
    and
    $$
   Z^\dagger_{n,p,\stheta} =O_\mathbb{P}(1) \max_{1\leq b \leq B} \|\bG_{\true}\|_2.
    $$
From \citet{bentkus2003dependence}, we can control the difference between the cumulative distribution function of $\bG_{\true}$ and the cumulative distribution function of a $p$-dimensional standard Gaussian random variable. Let
$$\Delta_n= \sup_{A \in \mathcal{C}} |\mathbb{P}(\bG_{\true}\in A) - \bnu(A)|,$$
where $\mathcal{C}$ is the class of convex subsets of $\mathbb{R}^p$ and $\bnu$ is the standard $p$ dimensional normal distribution, then  \citet{bentkus2003dependence} states that we have
 $$
 \Delta_n \leq 400 p^{1/4} \mathbb{E}[\|\bSigma_{\model,p}^{-1/2}\Psi_{\model,p}(\bX;\stheta) \|_2^3] n_b^{-1/2}.
 $$
 Using our assumptions, we have
 $$
 \Delta_n = O(n^{-\rho/2 + \kappa(7/4 + 3r_0/q_0)}).
 $$
 Hence, we have
 $$
 \sup_{t\in\mathbb{R}^+} | \mathbb{P}(\|\bG_{\true}\|^2_2<t)- F_{\mathcal{X}^2_p}(t) | =  O(n^{-\rho/2 + \kappa(7/4 + 3r_0/q_0)}),
 $$
where $F_{\mathcal{X}^2_p}$ is the cumulative distribution function of chi-square random variable with $p$ degrees of freedom. Hence, by independence between the observations
 $$
 \sup_{t\in\mathbb{R}^+} | \mathbb{P}(\max_{1\leq b\leq B} \|\bG_{\trueb}\|_2^2<t)- F^\dagger_{B,p}(t) | =  O(B n^{-\rho/2 + \kappa(7/4 + 3r_0/q_0)}),
 $$
 where $F_{\mathcal{X}^2_p}$ is the cumulative distribution function of the maximum of $B$ independent chi-square random variables with $p$ degrees of freedom each.
 Note that by Assumption~\ref{ass:main}-\ref{ass:growing1} we have $B n^{-\rho/2 + \kappa(7/4 + 3r_0/q_0)}=O(n^{1-3\rho/2 + \kappa(7/4 + 3r_0/q_0)})$ and thus, since $n^{1-3\rho/2 + \kappa(7/4 + 3r_0/q_0)}$ tends to zero by Assumptions~\ref{ass:main}-\ref{ass:growing3}, the approximation of $\max_{1\leq b\leq B} \|G_{\true}\|_2^2$ by a maximum of $B$ independent chi-square random variables with $p$ degrees of freedom is valid. In addition, the stochastic order of the maximum of $B$ independent chi-square random variables with $p$ degrees of freedom is of order is $O_\mathbb{P}(p + \ln B)$. Therefore, we have
 $$
 \max_{1\leq b \leq B} \|\bG_{\true}\|_2^2 = O_\mathbb{P}(n^{\kappa} + \ln n),
 $$
 leading that
 $$
  Z^\dagger_{n,p,\stheta}= O_\mathbb{P}(n^{\kappa/2} + \ln^{1/2} n).
 $$
 Hence, combining \eqref{eq:oderapproxZ} and Assumption \ref{ass:main}-\ref{ass:growing1}, we have
 $$
  Z^\dagger_{n,p,\wtheta}= O_\mathbb{P}(n^{\kappa/2} + \ln^{1/2} n).
 $$
 For any $\btheta$, define
 
$$ S^\dagger_{n,p,\btheta}=\left(\bSigma_{\model,p}-\bS_{\anyb} \right)^{\dagger},$$
we have
$$ S^\dagger_{n,p,\btheta} \leq p \upsilon^{\dagger}_{\anyb}.$$ 
Recall that $\widetilde{\upsilon}_{n,b} = |\upsilon_{\est} - \upsilon_{\true}|$, then we have by triangular inequality
$$
 \upsilon_{\estall}^{\dagger} \leq   \upsilon_{\trueb}^{\dagger} +  \widetilde{\upsilon}_{n}^{\dagger}.
$$
We have
$$
\widetilde{\upsilon}_{n}^{\dagger}  \leq \max_{p,j}\max_{1\leq i\leq n} | \psi_{\model,\wtheta,\varphi_{p,j}}(\bX_i) -  \psi_{\model,\stheta,\varphi_{p,j}}(\bX_i)|,
$$
leading by Assumptions~\ref{ass:main}-\ref{ass:cvtheta2} that
$$
p\widetilde{\upsilon}_{n}^{\dagger}=O_\mathbb{P}(n^{-\tau+\kappa}).
$$
Using the union bound and \eqref{eq:bounddiffsigma}, 
we have
$$
\mathbb{P}(p\upsilon_{\trueb}^{\dagger} \geq \varepsilon) \leq B\frac{C_2}{\varepsilon^{q_0}}  n^{-q_0 \vartheta_1}.
$$
Therefore, there exists a positive constant $C_3$ such that
$$
\mathbb{P}(p\upsilon_{\trueb}^{\dagger}   \geq \varepsilon) \leq  \frac{C_3}{\varepsilon^{q_0}}  n^{-q_0 \vartheta_3},
$$
where $\vartheta_3 = \vartheta_1 + \rho/q_0 - 1/q_0$ leading that
$$
p\upsilon_{\trueb}^{\dagger} = O_\mathbb{P}(n^{-\vartheta_3}).
$$
Noting that $\vartheta_3<\tau-\kappa$, we have $ p \upsilon_{\estall}^{\dagger}=O_\mathbb{P}(n^{-\vartheta_3})$ and thus
$$
S^\dagger_{n,p,\wtheta} =O_\mathbb{P}(n^{ -\vartheta_3}).
$$
By triangular inequality, we have
 $$
\bGamma^{\dagger}_{\estall} \leq \frac{\left( \bSigma_{\model,p}-\bS_{\estall}\right)^{\dagger}}{ \sigma_1(\bSigma_p)\min_{1 \leq b \leq B} \sigma_{1}(\bS_{\est})},
 $$
By Assumption~\ref{ass:main}-\ref{ass:cov}, we have $1/\sigma_1(\bSigma_p)=O(1)$. In addition, by Assumption~\ref{ass:main}-\ref{ass:growing3}, $\vartheta_3>0$ leading that  $S^\dagger_{n,p,\wtheta} =o_\mathbb{P}(1)$ and thus  $1/\min_{1 \leq m \leq B} \sigma_{1}(\bS_{\est})=1/\sigma_1(\bSigma_p)+O_\mathbb{P}(1)$. Therefore,
 $$
\bGamma^{\dagger}_{\estall}= O_\mathbb{P}\left(S^\dagger_{n,p,\wtheta}  \right).
$$
\end{proof}
\end{appendices}
 
\end{document}